\documentclass[final,onefignum,onetabnum]{siamart250211}

\usepackage{amsmath,amsfonts,amsopn,amssymb}
\usepackage{graphicx}
\usepackage[caption=false]{subfig}
\usepackage{multirow,relsize,makecell}
\usepackage{url}
\ifpdf
\DeclareGraphicsExtensions{.eps,.pdf,.png,.jpg}
\else
\DeclareGraphicsExtensions{.eps}
\fi
\usepackage{algorithm, algorithmic}
\usepackage[shortlabels]{enumitem}
\usepackage{accents}

\numberwithin{theorem}{section}
\newsiamremark{assumption}{Assumption}
\newsiamremark{remark}{Remark}
\newsiamremark{question}{Question}
\newsiamremark{example}{Example}
\crefname{assumption}{Assumption}{Assumptions}
\crefname{remark}{Remark}{Remarks}
\crefname{example}{Example}{Examples}

\headers{Subspace correction for semicoercive convex}{Young-Ju Lee and Jongho Park}

\title{Parallel subspace correction methods for semicoercive and nearly semicoercive convex optimization with applications to nonlinear PDEs\thanks{Submitted to arXiv.
\funding{Young-Ju Lee's work was supported by NSF-DMS 2208499.}
}}

\author{
Young-Ju Lee\thanks{Department of Mathematics, Texas State University, San Marcos, TX 78666, USA
  (\email{yjlee@txstate.edu}).}
\and
Jongho Park\thanks{Applied Mathematics and Computational Sciences Program, Computer, Electrical and Mathematical Science and Engineering Division, King Abdullah University of Science and Technology~(KAUST), Thuwal 23955, Saudi Arabia
 (\email{jongho.park@kaust.edu.sa}).}
}

\ifpdf
\hypersetup{
  pdftitle={Subspace correction methods for semicoercive and nearly semicoercive convex optimization with applications to nonlinear PDEs},
  pdfauthor={Young-Ju Lee, Jongho Park, and Jinchao Xu}
}
\fi

\usepackage[normalem]{ulem}
\usepackage{color}

\begin{document}

\maketitle
\begin{abstract}
We present new convergence analyses for parallel subspace correction methods for unconstrained semicoercive and nearly semicoercive convex optimization problems, generalizing the theory of singular and nearly singular linear problems to a class of nonlinear problems.
Our results demonstrate that the elegant theoretical framework developed for singular and nearly singular linear problems can be extended to unconstrained semicoercive and nearly semicoercive convex optimization problems.
For semicoercive problems, we show that the convergence rate can be estimated in terms of a seminorm stable decomposition over the subspaces and the kernel of the problem, aligning with the theory for singular linear problems.
For nearly semicoercive problems, we establish a parameter-independent convergence rate, assuming the kernel of the semicoercive part can be decomposed into a sum of local kernels, which aligns with the theory for nearly singular problems.
To demonstrate the applicability of our results, we provide convergence analyses of two-level additive Schwarz methods for solving certain nonlinear partial differential equations with Neumann boundary conditions, within the proposed abstract framework.
\end{abstract}



\section{Introduction}
\label{Sec:Introduction}
Many important linear problems arise in science and engineering as either singular or nearly singular. These problems can be characterized as systems, which have a nontrivial null space or near null space and they appear in various applications, which include finite element discretizations of the Poisson equation with pure Neumann boundary conditions, and/or variational problems of $H(\operatorname{div})$ and $H(\operatorname{curl})$~\cite{AFW:2000,BL:2005}. An important class of nearly singular problems can also be found at nearly incompressible linear elastic equations~\cite{LWC:2009}. Nearly singular problems also occur when solving indefinite systems arising from mixed finite element discretizations of the Navier--Stokes equations~\cite{FMW:2019}, as well as in more complex systems such as non-Newtonian fluids~\cite{LX:2006} and fluid--structure interaction problems~\cite{XY:2015}. In particular, the nearly incompressible linear elasticity problem arises as a subproblem to be solved \cite{LWC:2009},
when the augmented Lagrangian Uzawa method is employed~\cite{LWXZ:2007}.  

Due to the significance of singular and nearly singular linear problems in scientific computing, as highlighted by the numerous examples discussed above, there has been extensive research on numerical methods for solving these problems.
The theory of basic iterative methods for singular problems was first introduced in~\cite{Keller:1965}, with more refined results presented in later works such as~\cite{Cao:2008,LWXZ:2006,WLXZ:2008}.
In addition, subspace correction methods~\cite{Xu:1992,XZ:2002}, which offer a general framework for a variety of iterative methods, ranging from basic methods to advanced ones like domain decomposition and multigrid methods, were rigorously analyzed for singular linear problems in~\cite{LWXZ:2008,WLXZ:2008} and for nearly singular linear problems in~\cite{LWXZ:2007,WZ:2014}. Building on the theory of subspace correction methods, several applications have been developed for the specific examples discussed above; see, for instance,~\cite{FMW:2019,LWC:2009,XY:2015}.

A natural generalization of the concept of singularity in linear problems to convex optimization problems is semicoercivity~\cite{AET:2001,GG:2002}.
Intuitively, a convex functional is said to be semicoercive if it is flat along some subspace and increases to infinity in the other directions; the rigorous definition of semicoercivity can be found in~\cite{AET:2001}, as well as in \cref{Sec:Seminorm}.
In the special case of quadratic functionals, semicoercivity is equivalent to the singularity of the corresponding linear problem.
As a result, semicoercive problems, like singular ones, arise frequently in various nonlinear applications.
Consequently, there has been some research on efficient numerical solvers for particular semicoercive problems, such as~\cite{DFH:2023,DHS:2007}.

The goal of this paper is to extend the well-established theory of subspace correction methods for singular and nearly singular linear problems~\cite{LWXZ:2007,LWXZ:2008,WLXZ:2008,WZ:2014} to semicoercive and nearly semicoercive convex optimization problems.
This extension is motivated by recent developments in the convergence theory of subspace correction methods for convex optimization.
The framework of subspace correction methods for convex optimization was originally introduced in the foundational works~\cite{TE:1998,TX:2002}, which developed an abstract convergence theory and applied it to problems such as the $s$-Laplacian.
Other notable early results include~\cite{BK:2012,BTW:2003,Carstensen:1997}, while more refined and general convergence theories have been developed recently in~\cite{Park:2020,Park:2022a,PX:2024}.
These theories have also been successfully applied to a variety of nonlinear partial differential equations (PDEs) and variational inequalities~\cite{BF:2024,LP:2024,Park:2024b,Park:2024a}.
Despite these developments, there has been no prior work addressing subspace correction methods for semicoercive or nearly semicoercive problems. All existing theoretical analyses assume coercivity or stronger conditions such as uniform convexity.

In this paper, we present new convergence analyses for subspace correction methods for semicoercive and nearly semicoercive convex optimization problems in Banach spaces.
Roughly speaking, the main results of this paper combine the theories for singular and nearly singular problems~\cite{LWXZ:2007,LWXZ:2008,WZ:2014} with those for convex optimization problems~\cite{Park:2020,PX:2024} for subspace correction methods.
For semicoercive problems, we prove that the convergence rate of a subspace correction method can be estimated in terms of a seminorm stable decomposition over the subspaces and the kernel of the problem.
This aligns with the theory for singular linear problems established in~\cite{LWXZ:2008,WLXZ:2008}.
The analysis is analogous to that for coercive problems, except that special care must be taken with the kernel of the problem.
For nearly semicoercive problems, in the same spirit as in~\cite{LWXZ:2007}, we obtain a parameter-independent convergence rate estimate for subspace correction methods, under the assumption that the kernel of the semicoercive part, i.e., the near null space, can be decomposed into a sum of local near null spaces.
We note that the analysis of nearly singular linear problems in~\cite{LWXZ:2007} relies heavily on orthogonality in Hilbert spaces.
However, such orthogonality does not directly apply in our setting, as we are dealing with nonlinear problems posed in Banach spaces.
Therefore, we carefully extend the theory of nearly singular linear problems to nearly semicoercive convex optimization problems by employing nonlinear orthogonal decompositions in Banach spaces, which were introduced in~\cite{Alber:2000,Alber:2005}.
As examples of applications of the convergence theory developed in this paper, we provide convergence analyses of two-level additive Schwarz methods for solving a Neumann boundary value problem involving the $s$-Laplacian~\cite{LP:2024,Loisel:2020} and a related problem.

The rest of this paper is organized as follows.
In \cref{Sec:Seminorm}, we review the notion of semicoercivity and provide a characterization in terms of seminorms.
In \cref{Sec:SC}, we review subspace correction methods for convex optimization.
In \cref{Sec:Semicoercive}, we present a convergence analysis of subspace correction methods for semicoercive convex optimization.
In \cref{Sec:Nearly}, we present a convergence analysis for subspace correction methods for nearly semicoercive convex optimization.
In \cref{Sec:Applications}, we present several applications of the proposed convergence theory to domain decomposition methods for nonlinear PDEs.
In \cref{Sec:Conclusion}, we conclude the paper with remarks.

\section{Semicoercive functionals}
\label{Sec:Seminorm}
In this section, we present the notion of semicoercive functionals~\cite{AET:2001,GG:2002} and their characterization in terms of seminorms.
Additionally, we explore the relation between semicoercive convex functionals and semidefinite linear problems~\cite{LWXZ:2006,WLXZ:2008}.

We first recall the definition of semicoercive functionals introduced in~\cite{AET:2001}.
Let $V$ be a reflexive Banach space equipped with a norm $\| \cdot \|$. 
A proper functional $F \colon V \rightarrow \overline{\mathbb{R}}$ is said to be \textit{semicoercive} if there exists a closed subspace $\mathcal{N}$ of $V$ such that
\begin{equation}
\label{Def1:semicoercive}
F(v) = F(v + \phi)
\quad \forall v \in V, \text{ } \phi \in \mathcal{N},
\end{equation}
and the quotient functional $\bar{F} \colon V / \mathcal{N} \rightarrow \overline{\mathbb{R}}$ defined by
\begin{equation*}
\bar{F} (v + \mathcal{N}) = F (v),
\quad v \in V,
\end{equation*}
is coercive in the sense that
\begin{equation}
\label{Def2:semicoercive}
\bar{F} (v + \mathcal{N}) \rightarrow \infty 
\quad \text{ as } \| v + \mathcal{N} \|_{V/\mathcal{N}} \rightarrow \infty,
\end{equation}
where $\| \cdot \|_{V/\mathcal{N}}$ denotes the quotient norm given by
\begin{equation*}
\| v + \mathcal{N} \|_{V / \mathcal{N}} = \inf_{\phi \in \mathcal{N}} \| v + \phi \|,
\quad v \in V.
\end{equation*}
If such a subspace $\mathcal{N}$ exists, then it can be easily shown that $\mathcal{N}$ is unique.
We refer to this subspace as the \textit{kernel} of $F$, denoted by $\mathcal{N} = \ker F$. 

As described above, semicoercivity is defined in terms of a quotient space.
However, when developing the convergence theory for subspace correction methods and its applications, it is more convenient to work with seminorms rather than quotient spaces.
In \cref{Lem:seminorm}, we present a condition under which a seminorm can be characterized as a quotient norm.
Given a seminorm $| \cdot |$ on $V$, we define the \textit{kernel} $\ker | \cdot |$ of $| \cdot |$ as~(cf.~\cite{GG:2002})
\begin{equation*}
\ker | \cdot | = \left\{ v \in V : | v | = 0 \right\}.
\end{equation*}

\begin{lemma}
\label{Lem:seminorm}
Let $V$ be a Banach space.
A seminorm $| \cdot |$ on $V$ is equivalent to the quotient norm $\| \cdot + \mathcal{N} \|_{V / \mathcal{N}}$ for some closed subspace $\mathcal{N}$ of $V$ if and only if it satisfies the following:
\begin{enumerate}[(i)]
\item $| \cdot |$ is continuous on $V$.
\item There exists a positive constant $C$ such that
\begin{equation}
\label{Lem1:seminorm}
\inf_{\phi \in \ker |\cdot|} \| v + \phi \| \leq C | v |
\quad \forall v \in V.
\end{equation}
\end{enumerate}
\end{lemma}
\begin{proof}
Suppose that we have a seminorm $| \cdot |$ on $V$ that is equivalent to the quotient norm $\| \cdot + \mathcal{N} \|_{V/ \mathcal{N}}$ for some closed subspace $\mathcal{N}$ of $V$.
It follows directly that $\mathcal{N} = \ker | \cdot | $, making both~(i) and~(ii) straightforward.

Conversely, let $| \cdot |$ be a seminorm on $V$ that satisfies (i) and~(ii).
We set $\mathcal{N} = \ker | \cdot |$.
Due to the continuity of $| \cdot |$, we have
\begin{equation}
\label{Lem2:seminorm}
|v | = | v + \phi | \leq C' \| v + \phi \|
\end{equation}
for any $\phi \in \mathcal{N}$, where $C'$ is a positive constant.
By combining~\eqref{Lem1:seminorm} and~\eqref{Lem2:seminorm}, we deduce that $| \cdot |$ is equivalent to $\| \cdot + \mathcal{N} \|_{V / \mathcal{N}}$.
\end{proof}

\begin{remark}
\label{Rem:BH}
The condition~\eqref{Lem1:seminorm} is satisfied by many seminorms commonly encountered in PDEs.
For instance, the Bramble--Hilbert lemma~\cite{BH:1970} ensures that if $V = W^{m,p} (\Omega)$, where $m \in \mathbb{Z}_{> 0}$, $p \in [1, \infty]$, and the domain $\Omega \subset \mathbb{R}^d$ satisfies certain geometric conditions~\cite{DL:2004,DS:1980}, then~\eqref{Lem1:seminorm} holds for the Sobolev seminorm $| \cdot |_{W^{m,p} (\Omega)}$.
\end{remark}

Thanks to \cref{Lem:seminorm}, we are able to characterize semicoercivity in terms of seminorms; see \cref{Prop:semicoercive}.

\begin{proposition}
\label{Prop:semicoercive}
Let $V$ be a reflexive Banach space. A proper functional $F \colon V \rightarrow \overline{\mathbb{R}}$ is semicoercive if and only if there exists a continuous seminorm $| \cdot |$ on $V$ that satisfies~\eqref{Lem1:seminorm} and the following:
\begin{enumerate}[(i)]
\item $F(v) = F(v + \phi)$ for any $v \in V$ and $\phi \in \ker | \cdot |$.
\item $F(v) \rightarrow \infty$ as $| v | \rightarrow \infty$.
\end{enumerate}
\end{proposition}
\begin{proof}
If we have a semicoercive functional $F$, then the seminorm $| \cdot |$ defined as
\begin{equation*}
    |v| = \inf_{\phi \in \ker F} \| v + \phi \|,
    \quad v \in V,
\end{equation*}
is continuous and satisfies~\eqref{Lem1:seminorm}, (i) and~(ii).
Conversely, if we have a continuous seminorm $| \cdot |$ on $V$ that satisfies~\eqref{Lem1:seminorm}, (i), and~(ii), then we can readily deduce, using \cref{Lem:seminorm}, that $F$ is semicoercive with the kernel $\ker F = \ker | \cdot |$.
\end{proof}

We conclude this section by demonstrating that minimizing semicoercive and convex energy functionals generalizes solving semidefinite linear problems.
Let $H$ be a Hilbert space equipped with an inner product $( \cdot, \cdot )$.
We consider the following semidefinite linear problem:
\begin{equation}
\label{semidefinite}
A u = f,
\end{equation}
where $A \colon H \rightarrow H$ is a continuous, symmetric and positive semidefinite linear operator satisfying the condition~(cf.~\cite[Equation~(2.3)]{LWXZ:2008})
\begin{equation*}
(A v, v) \geq \mu \| v + \operatorname{ker} A \|_{V  / \operatorname{ker} A}^2, \quad v \in V,
\end{equation*}
for some $\mu > 0$, and $f \in \operatorname{ran} A$.
It is straightforward to verify that $u \in H$ solves~\eqref{semidefinite} if and only if it minimizes the following quadratic energy functional:
\begin{equation}
\label{semidefinite_energy}
F(v) = \frac{1}{2} ( Av, v ) - ( f, v ),
\quad v \in H.
\end{equation}
That is, the semidefinite linear problem~\eqref{semidefinite} is equivalent to the minimization problem given by~\eqref{semidefinite_energy}.
Clearly, the energy functional $F$ in~\eqref{semidefinite_energy} is convex and semicoercive with the kernel $\ker F = \ker A$.
Hence, we conclude that the semidefinite linear problem~\eqref{semidefinite} is a special case of semicoercive convex optimization.

\section{Subspace correction methods for convex optimization}
\label{Sec:SC}
In this section, we briefly summarize subspace correction methods for convex optimization, which have been extensively studied in the literature, e.g.,~\cite{Park:2020,PX:2024,TE:1998,TX:2002}.
For simplicity, we focus on the case of exact local problems only; the case of inexact local problems~\cite{Park:2020,PX:2024} will be considered in \cref{App:Local}.

We consider the following abstract convex optimization problem on a reflexive Banach space $V$:
\begin{equation}
\label{model}
    \min_{v \in V} F(v),
\end{equation}
where $F \colon V \rightarrow \mathbb{R}$ is a G\^{a}teaux differentiable and convex functional.
We can readily verify that the problem~\eqref{model} admits a solution~(may not be unique) $u \in V$ if $F$ is semicoercive.

We assume that the solution space $V$ of~\eqref{model} admits a space decomposition
\begin{equation}
\label{space_decomp}
    V = \sum_{j=1}^N V_j,
\end{equation}
where each $V_j$, $j \in [N] = \{ 1, 2, \dots, N \}$, is a closed subspace of $V$.
One important property of the space decomposition~\eqref{space_decomp} is the stable decomposition property.
Namely, we have
\begin{equation}
\label{norm_stable}
\sup_{\| w \| = 1} \inf_{\sum_{j=1}^N w_j = w} \left( \sum_{j=1}^N \| w_j \|^q \right)^{\frac{1}{q}} < \infty,
\end{equation}
for any $q \in [1, \infty)$, where $w$ and $w_j$ are taken from $V$ and $V_j$, respectively.
This property follows directly from the open mapping theorem~(see~\cite[Equation~(2.15)]{XZ:2002} and~\cite[Equation~(2)]{PX:2024}).

Meanwhile, the convexity of the energy functional $F$ in~\eqref{model} implies the following inequality, known as the strengthened convexity condition~\cite[Assumption~4.2]{Park:2020}, holds for some $\tau > 0$:
\begin{equation}
\label{convexity}
(1 - \tau N) F(v) + \tau \sum_{j=1}^N F(v+ w_j) \geq F \left( v + \tau \sum_{j=1}^N w_j \right), \text{ } v \in V, \text{ } w_j \in V_j.
\end{equation}
A positive constant $\tau_0$ is defined as the maximum $\tau$ that satisfies the strengthened convexity condition, i.e.,
\begin{equation}
\label{tau_0}
\tau_0 = \max \left\{ \tau > 0 : \text{The inequality~\eqref{convexity} holds} \right\}.
\end{equation}
Then it is clear that~\eqref{convexity} holds for every $\tau > 0$ less than or equal to $\tau_0$.  
While we have a trivial estimate $\tau_0 \geq 1/N$, in many applications, better estimates for $\tau_0$, which are often independent of $N$, can be obtained using a coloring argument~\cite{Park:2020,TX:2002}.
One may refer to~\cite[Section~4.1]{Park:2020} for a discussion on the relation between~\eqref{convexity} and strengthened Cauchy--Schwarz inequalities, which plays a crucial role in the analysis of multilevel methods for linear problems~\cite{TW:2005,Xu:1992}.

Subspace correction methods involve local problems defined in the subspaces $\{ V_j \}_{j=1}^N$.
Given $v \in V$, the optimal residual in the subspace $V_j$, which the energy functional $F$, is obtained by solving the minimization problem
\begin{equation}
\label{local_exact}
\min_{w_j \in V_j} \left\{ F_j (w_j; v) := F(v + w_j) \right\}.
\end{equation}
The parallel subspace correction method, also known as the additive Schwarz method in the literature on domain decomposition methods, for solving~\eqref{model}  with the local problem~\eqref{local_exact} is presented in \cref{Alg:PSC}.
The upper bound $\tau_0$ for the step size $\tau$ was given in~\eqref{tau_0}.
We note that \cref{Alg:PSC} was originally introduced in~\cite{TE:1998,TX:2002} and 
has since been extensively studied; see~\cite{Park:2020} and the references therein.

\begin{algorithm}
\caption{Parallel subspace correction method}
\begin{algorithmic}[]
\label{Alg:PSC}
\STATE Given the step size $\tau \in (0, \tau_0]$:
\STATE Choose $u^{(0)} \in V$.
\FOR{$n=0,1,2,\dots$}
    \FOR{$j \in [N]$ \textbf{in parallel}}
        \STATE $\displaystyle
            w_j^{(n+1)} \in \operatornamewithlimits{\arg\min}_{w_j \in V_j} F_j (w_j; u^{(n)})
        $
    \ENDFOR
    \STATE $\displaystyle
    u^{(n+1)} = u^{(n)} + \tau \sum_{j=1}^N w_j^{(n+1)}
    $
\ENDFOR
\end{algorithmic}
\end{algorithm}

\begin{remark}
\label{Rem:acceleration}
Incorporating acceleration schemes designed for first-order methods for convex optimization~(see, e.g.,~\cite{Nesterov:2013}) into \cref{Alg:PSC} yields accelerated variants of the parallel subspace correction method~\cite{Park:2021,Park:2022a}.
These accelerated methods typically exhibit faster convergence compared to the unaccelerated version, with only a marginal increase in computational cost per iteration.
However, a detailed discussion of these accelerated methods is beyond the scope of this paper.
\end{remark}

Another type of subspace correction method, known as the successive subspace correction method (also referred to as the multiplicative Schwarz method), involves solving the local problems sequentially.
In this paper, however, we focus on the parallel subspace correction method.



\section{Convergence analysis for semicoercive problems}
\label{Sec:Semicoercive}
In this section, we present a new convergence theory for subspace correction methods for semicoercive convex optimization.
Throughout this section, we assume that the energy functional $F$ in the problem~\eqref{model} is semicoercive with respect to a seminorm $| \cdot |$ in the sense of \cref{Prop:semicoercive}, and denote $\mathcal{N} = \ker F = \ker | \cdot |$.
To this end, we derive convergence rate estimates for the parallel subspace correction in terms of a seminorm stable decomposition over the subspaces $\{ V_j \}$ and the kernel $\mathcal{N}$. 
These results align with the sharp theory of singular linear problems established in~\cite{LWXZ:2008,WLXZ:2008}.

\begin{remark}
\label{Rem:local_kernel}
Even if the energy functional $F$ is semicoercive with a nontrivial kernel, the kernel of each local problem in subspace correction methods may still be trivial, i.e., $V_j \cap \ker F = \{ 0 \}$.
\end{remark}

\subsection{Descent property}
We first show that the parallel subspace correction method for solving the semicoercive problem~\eqref{model} achieves a certain descent property on the energy.
This descent property will play a central role in analyzing the convergence rate.

Let $V^*$ be the topological dual space of the reflexive Banach space $V$.
Given $v \in V$ and $w_j \in V_j$, we denote the Bregman distance associated with $F$ between $v + w_j$ and $v$ by $d_j (w_j; v)$, i.e.,
\begin{equation}
\label{d_j}
d_j (w_j; v) = F_j (w_j; v) - F(v) - \langle F'(v), w_j\rangle,
\quad v\in V, \text{ } w_j \in V_j,
\end{equation}
where $F'(v) \in V^*$ is the G\^{a}teaux derivative of $F$ at $v$, and $\langle \cdot, \cdot \rangle$ represents the duality pairing on $V$.

In \cref{Lem:ASM}, we present a generalized additive Schwarz lemma~\cite{LP:2024,Park:2020,PX:2024} for semicoercive problems, which states that the parallel subspace correction method can be viewed as a gradient descent method endowed with a specific nonlinear metric-like functional.

\begin{lemma}[generalized additive Schwarz lemma]
\label{Lem:ASM}
Let $V$ be a reflexive Banach space, and let $F \colon V \rightarrow \mathbb{R}$ be a G\^{a}teaux differentiable, convex, and semicoercive functional with the kernel $\mathcal{N}$.
For $v \in V$, we have
\begin{equation}
\label{Lem1:ASM}
\hat{w} :=  \sum_{j=1}^N \hat{w}_j \in \operatornamewithlimits{\arg\min}_{w \in V} \left\{ \langle F'(v) ,w \rangle + \inf_{\phi \in \mathcal{N}} \inf_{\sum_{j=1}^N w_j = w + \phi} \sum_{j=1}^N d_j (w_j; v) \right\}, 
\end{equation}
where $\hat{w}_j \in V_j$, $j \in [N]$, is given by
\begin{equation}
\label{hat_w}
\hat{w}_j \in \operatornamewithlimits{\arg\min}_{w_j \in V_j} F(w_j; v)
= \operatornamewithlimits{\arg\min}_{w_j \in V_j} \left\{ \langle F'(v), w_j \rangle + d_j (w_j; v) \right\}.
\end{equation}
Moreover, we have
\begin{equation}
\label{Lem2:ASM}
\inf_{\phi \in \mathcal{N}} \inf_{\sum_{j=1}^N w_j = \hat{w} + \phi} \sum_{j=1}^N d_j (w_j; v)
= \sum_{j=1}^N d_j(\hat{w}_j; v).
\end{equation}
\end{lemma}
\begin{proof}
Throughout the proof, we write
\begin{equation*}
d(w; v) = \inf_{\phi \in \mathcal{N}} \inf_{\sum_{j=1}^N w_j = w + \phi} d_j (w_j; v),
\quad w \in V.
\end{equation*}
We take any $w \in V$.
For any $\phi \in \mathcal{N}$ and $w_j \in V_j$, $j \in [N]$, such that $w = \sum_{j=1}^N w_j + \phi$, we get
\begin{equation}
\label{Lem3:ASM}
\begin{split}
\langle F'(v), \hat{w} \rangle + d(\hat{w}; v)
&\leq \sum_{j=1}^N \left( \langle F'(v), \hat{w}_j \rangle + d_j (\hat{w}_j; v) \right) \\
&\stackrel{\eqref{hat_w}}\leq \sum_{j=1}^N \left( \langle F'(v), w_j \rangle + d_j (w_j; v) \right) \\
&= \langle F'(v), w \rangle + \sum_{j=1}^N d_j (w_j; v).
\end{split}
\end{equation}
where the first inequality holds because $\hat{w} = 0 + \sum_{j=1}^N \hat{w}_j \in \mathcal{N} + \sum_{j=1}^N V_j$, and the last equality follows from $\langle F'(v), \phi \rangle = 0$.
By minimizing the last line of~\eqref{Lem3:ASM} over all $(w_j)_{j=1}^N$ and $\phi$, we obtain
\begin{equation}
\label{Lem4:ASM}
\langle F'(v), \hat{w} \rangle + d( \hat{w}; v)
\leq \sum_{j=1}^N \left( \langle F'(v), \hat{w}_j \rangle + d_j (\hat{w}_j; v) \right)
\leq \langle F'(v), w \rangle + d(w; v),
\end{equation}
which implies~\eqref{Lem1:ASM}.
Then, setting $w = \hat{w}$ in~\eqref{Lem4:ASM} yields~\eqref{Lem2:ASM}.
\end{proof}

Using \cref{Lem:ASM}, we can establish a descent property of the parallel subspace correction method for semicoercive problems, as presented in \cref{Lem:descent}.
We note that a corresponding result for coercive problems can be found in~\cite[Theorem~1]{PX:2024}.

\begin{lemma}
\label{Lem:descent}
Let $V$ be a reflexive Banach space, and let $F \colon V \rightarrow \mathbb{R}$ be a G\^{a}teaux differentiable, convex, and semicoercive functional with the kernel $\mathcal{N}$.
In \cref{Alg:PSC}, we have
\begin{equation*}
\resizebox{\hsize}{!}{%
$\displaystyle
F(u^{(n+1)}) \leq F(u^{(n)}) + \tau \min_{w \in V} \left\{ \langle F'(u^{(n)}), w \rangle + \inf_{\phi \in \mathcal{N}} \inf_{\sum_{j=1}^N w_j = w + \phi} \sum_{j=1}^N d_j (w_j; u^{(n)}) \right\},
\text{ } n \geq 0.
$
}
\end{equation*}
\end{lemma}
\begin{proof}
Take any $n \geq 0$.
Since the strengthened convexity condition~\eqref{convexity} implies
\begin{equation}
\label{Lem1:descent}
F(u^{(n+1)}) \leq (1 - \tau N) F(u^{(n)}) + \tau \sum_{j=1}^N F(u^{(n)} + w_j^{(n+1)}),
\end{equation}
it suffices to estimate the term $\sum_{j=1}^N F(u^{(n)} + w_j^{(n+1)})$.
It follows that
\begin{multline}
\label{Lem2:descent}
\sum_{j=1}^N F(u^{(n)} + w_j^{(n+1)})
= N F(u^{(n)}) + \sum_{j=1}^N \left( \langle F'(u^{(n)}), w_j^{(n+1)} \rangle + d_j (w_j^{(n+1)}; u^{(n)}) \right) \\
= N F(u^{(n)}) + \min_{w \in V } \left\{ \langle F'(u^{(n)}), w \rangle + \inf_{\phi \in \mathcal{N}} \inf_{\sum_{j=1}^N w_j = w + \phi} \sum_{j=1}^N d_j (w_j; u^{(n)}) \right\},
\end{multline}
where the last inequality is due to \cref{Lem:ASM}.
Combining~\eqref{Lem1:descent} and~\eqref{Lem2:descent} yields the desired result.
\end{proof}

A straightforward consequence of \cref{Lem:descent} is that the energy sequence generated by the parallel subspace correction method is decreasing; see \cref{Cor:descent}.

\begin{corollary}
\label{Cor:descent}
Let $V$ be a reflexive Banach space, and let $F \colon V \rightarrow \mathbb{R}$ be a G\^{a}teaux differentiable, convex, and semicoercive functional.
In \cref{Alg:PSC}, the energy sequence $\{ F(u^{(n)} ) \}$ is decreasing.
\end{corollary}

\subsection{Convergence rate analysis}
Now, we derive convergence rate estimates for the parallel subspace correction method for solving the semicoercive problem~\eqref{model} based on the descent property presented in \cref{Lem:descent}.

A key ingredient in the convergence rate analysis is a stable decomposition property of the seminorm $| \cdot |$, similar to~\eqref{norm_stable}.
To establish this property, we assume that the seminorm satisfies the conditions stated in~\cref{Lem:seminorm}, as well as the local condition described in \cref{Ass:seminorm_local}.
It is worth noting that a similar assumption was made in the theory of singular linear problems; see~\cite[Equation~(A1)]{LWXZ:2008}.

\begin{assumption}
\label{Ass:seminorm_local}
For each $j \in [N]$, there exists a positive constant $C_j$ such that
\begin{equation*}
\inf_{\phi_j \in V_j \cap \ker | \cdot |} \| v_j + \phi_j \| \leq C_j | v_j |
\quad \forall v_j \in V_j.
\end{equation*}
\end{assumption}

\begin{remark}
\label{Rem:seminorm_local}
As discussed in~\cite[Example~A.1]{LWXZ:2008}, \cref{Ass:seminorm_local} does not follow from~\eqref{Lem1:seminorm} in general, and needs to be added as an additional assumption.
In the same spirit as \cref{Rem:BH}, this condition is satisfied by many seminorms commonly encountered in PDEs.
\end{remark}

In \cref{Lem:seminorm_stable}, we show that, under \cref{Ass:seminorm_local}, the stable decomposition property holds with respect to the seminorm $| \cdot |$ as well if the kernel of the seminorm is included as an additional subspace.

\begin{lemma}
\label{Lem:seminorm_stable}
Let $| \cdot |$ be a continuous seminorm on a Banach space $V$ that satisfies~\eqref{Lem1:seminorm}.
In addition, suppose that \cref{Ass:seminorm_local} holds.
Then we have
\begin{equation*}
\sup_{| w | = 1} \inf_{\phi \in \ker | \cdot |} \inf_{\sum_{j=1}^N w_j = w + \phi} \left( \sum_{j=1}^N \| w_j \|^q \right)^{\frac{1}{q}} < \infty,
\end{equation*}
\end{lemma}
for any $q \in [1, \infty)$, where $w$ and $w_j$ are taken from $V$ and $V_j$, respectively.
\begin{proof}
We write $\mathcal{N} = \operatorname{ker} | \cdot |$.
Take any $w \in V$ and $\epsilon > 0$.
By~\eqref{Lem1:seminorm}, there exists $\tilde{\phi} \in \mathcal{N}$ such that
\begin{equation}
\label{Lem1:seminorm_stable}
\| w + \tilde{\phi} \|
\leq \inf_{\phi \in \mathcal{N}} \| w + \phi \| + \epsilon
\leq C | w | + \epsilon.
\end{equation}
Thanks to~\eqref{norm_stable}, there exist $\tilde{w}_j \in V_j$, $j \in [N]$, such that
\begin{equation}
\label{Lem2:seminorm_stable}
\sum_{j=1}^N \tilde{w}_j = w + \tilde{\phi}, \quad
\left( \sum_{j=1}^N \| \tilde{w}_j \|^q \right)^{\frac{1}{q}} \leq C' \| w + \tilde{\phi} \| + \epsilon,
\end{equation}
for some constant $C' > 0$.
Combining~\eqref{Lem1:seminorm_stable} and~\eqref{Lem2:seminorm_stable}, we obtain
\begin{equation*}
\inf_{\phi \in \mathcal{N}} \inf_{\sum_{j=1}^N w_j = w + \phi} \left( \sum_{j=1}^N \| w_j \|^q \right)^{\frac{1}{q}}
\leq \left( \sum_{j=1}^N \| \tilde{w}_j \|^q \right)^{\frac{1}{q}}
\leq C C' | w | + (C' + 1) \epsilon.
\end{equation*}
Since $w$ and $\epsilon$ are arbitrary, we obtain the desired result.
\end{proof}

Meanwhile, from \cref{Lem:descent}, we observe that the convergence behavior of the parallel subspace correction method depends on the infimum of the sum of the local functionals $d_j (w_j; v)$, $j \in [N]$.
To establish an upper bound for this infimum, we require a smoothness assumption on each $d_j (w_j; v)$, which ensures that $d_j (w_j; v)$ can be bounded above by a power of the norm $\| w_j \|^q$; see \cref{Ass:smooth}.
We note that, throughout this paper, we adopt the convention $0/0 = 0$ for arguments of $\sup$ and $0/0 = \infty$ for arguments of $\inf$.

\begin{assumption}[local smoothness]
\label{Ass:smooth}
For some $q > 1$, each $d_j (w_j; v)$, $j \in [N]$, satisfies the following:
for any $| \cdot |$-bounded convex subset $K \subset V$ and $\| \cdot \|$-bounded convex subset $K_j \subset V_j$ satisfying $0 \in K_j$, we have
\begin{equation*}
\sup_{v \in K, w_j \in K_j} \frac{d_j (w_j; v)}{\| w_j \|^q} < \infty.
\end{equation*}
\end{assumption}

In the case of exact local problems~\eqref{local_exact}, an easy-to-check sufficient condition for \cref{Ass:smooth} is the weak smoothness~\cite{Park:2022b} of the energy functional $F$, which is valid in many applications involving nonlinear PDEs~(see, e.g.,~\cite{CHW:2020,LP:2024,TX:2002}).
We summarize this result in \cref{Prop:smooth}.
In what follows, given $v, w \in V$, we denote the Bregman distance associated with $F$ between $v+w$ and $v$ by $d_F(w; v)$, i.e.,
\begin{equation}
\label{d_F}
d_F(w; v) = F(v + w) - F(v) - \langle F'(v), w \rangle,
\quad v, w \in V.
\end{equation}

\begin{proposition}
\label{Prop:smooth}
Let $V$ be a reflexive Banach space, and let $F \colon V \rightarrow \mathbb{R}$ be a G\^{a}teaux differentiable, convex, and semicoercive functional with respect to a seminorm $| \cdot |$ in the sense of \cref{Prop:semicoercive}.
Suppose that each $d_j (w_j; v)$, $j \in [N]$, is given by~\eqref{local_exact} and~\eqref{d_j}.
For some $q > 1$, if $F$ is locally $q$-weakly smooth, i.e., if
\begin{equation}
\label{Prop1:smooth}
\sup_{v, v+w \in K}
\frac{d_F(w; v)}{\| w \|^q} < \infty,
\end{equation}
for any $\| \cdot \|$-bounded convex subset $K$ of $V$, then \cref{Ass:smooth} holds.
\end{proposition}
\begin{proof}
Throughout this proof, let $C$ denote a general positive constant.
First, we prove that the local smoothness with respect to the norm $\| \cdot \|$ implies the local smoothness with respect to the seminorm $| \cdot |$ as well, i.e.,
\begin{equation}
\label{Prop2:smooth}
\sup_{v, v+w \in K} \frac{d_F(w; v)}{| w |^q} < \infty,
\end{equation}
for any $| \cdot |$-bounded convex subset $K$ of $V$.
Suppose that~\eqref{Prop1:smooth} holds, and take any $| \cdot |$-bounded convex subset $K$ of $V$.
We define the set $\widehat{K}$ as the convex hull of the following set:
\begin{equation*}
 \left\{ v + \phi :  v \in K, \text{ } \phi = \operatornamewithlimits{\arg\min}_{\psi \in \ker | \cdot | } \| v + \psi \| \right\}.
\end{equation*}
Then~\eqref{Lem1:seminorm} implies that $\widehat{K}$ is $\| \cdot \|$-bounded.
Hence, for any $v, v+w \in K$, we have
\begin{equation*}
d_F(w; v)
= d_F(\hat{w}; \hat{v})
\stackrel{\eqref{Prop1:smooth}}{\leq} C \| \hat{w} \|^q
\stackrel{\eqref{Lem1:seminorm}}{\leq} C | \hat{w} |^q
= C | w |^q,
\end{equation*}
where $\hat{v} \in \widehat{K}$ is defined as
\begin{equation*}
\hat{v} = v + \phi,
\quad 
\phi = \operatornamewithlimits{\arg\min}_{\psi \in \ker | \cdot | } \| v + \psi \|,
\end{equation*}
and $\hat{w} \in \widehat{K}$ is defined similarly.
This implies that~\eqref{Prop2:smooth} holds.

Now, it is enough to show that~\eqref{Prop2:smooth} implies \cref{Ass:smooth}.
Suppose that~\eqref{Prop2:smooth} holds.
Take any $j \in [N]$, $| \cdot |$-bounded convex subset $K$ of $V$, and $\| \cdot \|$-bounded convex subset $K_j$ of $V_j$ such that $0 \in K_j$.
It is straightforward to verify that $K + K_j$ is $| \cdot |$-bounded.
Moreover, if $v \in K$ and $w_j \in K_j$, then we have $v, v + w_j \in K + K_j$.
Applying~\eqref{Prop2:smooth} with the set $K + K_j$ yields the desired result.
\end{proof}

In \cref{Lem:stable}, we combine \cref{Lem:seminorm_stable} and \cref{Ass:smooth} to show that the infimum of the sum of the local functionals $d_j (w_j; v)$ appeared in \cref{Lem:descent} can be bounded above a power of the seminorm $| w |^q$.

\begin{lemma}[stable decomposition]
\label{Lem:stable}
Let $V$ be a reflexive Banach space, and let $F \colon V \rightarrow \mathbb{R}$ be a G\^{a}teaux differentiable, convex, and semicoercive functional with respect to a seminorm $| \cdot |$ in the sense of \cref{Prop:semicoercive}.
Suppose that \cref{Ass:seminorm_local,Ass:smooth} hold.
For any $| \cdot |$-bounded convex subset $K \subset V$, the following holds:
\begin{equation}
\label{C_K}
C_K := q \sup_{v, v+w \in K} \inf_{\phi \in \ker | \cdot |} \inf_{\sum_{j=1}^N w_j = w + \phi} \frac{\sum_{j=1}^N d_j (w_j; v)}{| w |^q} < \infty.
\end{equation}
\end{lemma}
\begin{proof}
We take any $| \cdot |$-bounded convex subset $K$ of $V$, and write
\begin{equation*}
M_K := \sup_{v \in K} |v| < \infty.
\end{equation*}
Choose any $v, v+w \in K$.
We first observe that
\begin{equation}
\label{Lem1:stable}
|w | \leq |v| + |v + w| \leq 2M_K.
\end{equation}
We write $\mathcal{N} = \operatorname{ker} | \cdot |$.
Let $( ( \tilde{w}_j )_{j=1}^N, \tilde{\phi}) \in \prod_{j=1}^N V_j \times \mathcal{N}$ be a minimizer of $\sum_{j=1}^N \| w_j \|^q$ over $w_j \in V_j$ and $\phi \in \mathcal{N}$ such that $\sum_{j=1}^N w_j = w + \phi$.
By \cref{Lem:seminorm_stable}, we have
\begin{equation}
\label{Lem2:stable}
\| \tilde{w}_j \|
\leq \left( \sum_{j=1}^N \| \tilde{w}_j \|^q \right)^{\frac{1}{q}}
\leq C | w |
\stackrel{\eqref{Lem1:stable}}{\leq} 2 C M_K,
\end{equation}
for some $C > 0$.
If we set $K_j = \{ w_j \in V_j : \| w_j \| \leq 2 C M_K \}$ for each $j \in [N]$ in \cref{Ass:smooth}, then we obtain
\begin{equation}
\label{Lem3:stable}
L_j := q \sup_{\bar{v} \in K, w_j \in K_j} \frac{d_j (w_j; \bar{v})}{\| w_j \|^q} < \infty.
\end{equation}
It follows that
\begin{multline*}
\inf_{\phi \in \mathcal{N}} \inf_{\sum_{j=1}^N w_j = w + \phi} \sum_{j=1}^N d_j (w_j; v)
\leq \sum_{j=1}^{N} d_j (\tilde{w}_j; v) \\
\stackrel{\eqref{Lem3:stable}}\leq \sum_{j=1}^N \frac{L_j}{q} \| \tilde{w}_j \|^q
\stackrel{\eqref{Lem2:stable}}{\leq} \frac{C^q}{q} \left( \max_{j \in [N]} L_j \right) | w |^q.
\end{multline*}
Consequently, we get
\begin{equation*}
\inf_{\phi \in \mathcal{N}} \inf_{\sum_{j=1}^N w_j = w + \phi} \frac{\sum_{j=1}^N d_j (w_j; v)}{| w|^q} \leq \frac{C^q}{q} \max_{j \in [N]} L_j.
\end{equation*}
Note that the right-hand side is independent of both $v$ and $w$.
Hence, we conclude that
\begin{equation}
\label{Lem4:stable}
C_K = q \sup_{v, v+w \in K} \inf_{\phi \in \mathcal{N}} \inf_{\sum_{j=1}^N w_j = w + \phi} \frac{\sum_{j=1}^N d_j (w_j; v)}{| w|^q}
\leq C^q \max_{j \in [N]} L_j < \infty,
\end{equation}
which completes the proof.
\end{proof}

Given the initial iterate $u^{(0)} \in V$ of \cref{Alg:PSC}, we define
\begin{subequations}
\label{KR_0}
\begin{align}
\label{K_0}
K_0 &= \{ v \in V : F(v) \leq F(u^{(0)}) \}, \\
\label{R_0}
R_0 &= \sup_{v \in K_0} | v - u |.
\end{align}
\end{subequations}
The convexity and semicoercivity of $F$ implies that $K_0$ is $| \cdot |$-bounded and convex, and that $R_0 < \infty$.
Moreover, \cref{Cor:descent} implies that the sequence $\{ u^{(n)} \}$ generated by \cref{Alg:PSC} is contained in $K_0$.
Consequently, thanks to \cref{Lem:stable}, we have
\begin{equation*}
\inf_{\phi \in \mathcal{N}} \inf_{\sum_{j=1}^N w_j = w + \phi} \sum_{j=1}^N d_j (w_j; u^{(n)})
\leq \frac{C_{K_0}}{q} | w |^q,
\end{equation*}
for any $w \in V$ such that $u^{(n)} + w \in K$ and $n \geq 0$, where $C_{K_0}$ was given in~\eqref{C_K}.
Using a similar argument as in~\cite{Park:2020,Park:2022a}, we are able to derive the following convergence theorem for the parallel subspace correction method for solving~\eqref{model}.
A proof of \cref{Thm:conv} can be found in \cref{App:Proofs}.

\begin{theorem}
\label{Thm:conv}
Let $V$ be a reflexive Banach space, and let $F \colon V \rightarrow \mathbb{R}$ be a G\^{a}teaux differentiable, convex, and semicoercive functional with respect to a seminorm $| \cdot |$ in the sense of \cref{Prop:semicoercive}.
Suppose that \cref{Ass:smooth} holds.
In \cref{Alg:PSC}, let $\zeta_n = F(u^{(n)}) - F(u)$ for $n \geq 0$.
If $\zeta_0 > C_{K_0} R_0^q$, then we have
\begin{equation*}
\zeta_1 \leq \left(1 - \tau \left( 1 - \frac{1}{q} \right) \right) \zeta_0,
\end{equation*}
where $C_{K_0}$ and $R_0$ were given in~\eqref{C_K} and~\eqref{KR_0}.
Otherwise, we have
\begin{equation*}
\zeta_n \leq \frac{C}{\left( n + (C/\zeta_0)^{1/\beta} \right)^{\beta}},
\quad n \geq 0,
\end{equation*}
where
\begin{equation*}
\beta = q-1, \quad
C = \left(\frac{q}{\tau}\right)^{q-1} C_{K_0} R_0^q.
\end{equation*}
\end{theorem}

In the same spirit as the convergence theory for coercive problems~\cite{Park:2020,Park:2022a,PX:2024}, we are able to obtain an improved convergence rate of the parallel subspace correction method if we additionally assume that the energy functional $F$ is sharp~\cite{RD:2020} around a minimizer.
We formally summarize this sharpness assumption in \cref{Ass:sharp}.

\begin{assumption}[sharpness]
\label{Ass:sharp}
For some $p > 1$, the function $F$ satisfies the following: for any $| \cdot |$-bounded convex subset $K \subset V$ satisfying $u \in K$, we have
\begin{equation}
\label{mu_K}
\mu_K := p \inf_{v \in K} \frac{F(v) - F(u)}{| v - u |^p} > 0.
\end{equation}
\end{assumption}

In \cref{Thm:conv_sharp}, we provide convergence rate estimates for the parallel subspace correction method under the additional assumption described in \cref{Ass:sharp}.
A proof of \cref{Thm:conv_sharp} is given in \cref{App:Proofs}.

\begin{theorem}
\label{Thm:conv_sharp}
Let $V$ be a reflexive Banach space, and let $F \colon V \rightarrow \mathbb{R}$ be a G\^{a}teaux differentiable, convex, and semicoercive functional with respect to a seminorm $| \cdot |$ in the sense of \cref{Prop:semicoercive}.
Suppose that \cref{Ass:smooth,Ass:sharp} hold.
In \cref{Alg:PSC}, let $\zeta_n = F(u^{(n)}) - F(u)$ for $n \geq 0$.
Then we have the following:
\begin{enumerate}[(a)]
\item In the case $p = q$, we have
\begin{equation*}
\zeta_n \leq \left( 1- \tau \left( 1 - \frac{1}{q} \right) \min \left\{ 1, \frac{\mu_{K_0}}{q C_{K_0}} \right\}^{\frac{1}{q-1}} \right)^n \zeta_0,
\quad n \geq 0,
\end{equation*}
where $C_{K_0}$ and $\mu_{K_0}$ were given in~\eqref{C_K}, \eqref{K_0}, and~\eqref{mu_K}.

\item In the case $p > q$, if $\zeta_0 > \left( \frac{p}{\mu_{K_0}} \right)^{\frac{q}{p-q}} C_{K_0}^{\frac{p}{p-q}}$, then we have
\begin{equation*}
\zeta_1 \leq \left(1 - \tau \left( 1 - \frac{1}{q} \right) \right) \zeta_0.
\end{equation*}
Otherwise, we have
\begin{equation*}
\zeta_n \leq \frac{C}{\left( n + (C/\zeta_0)^{1/\beta} \right)^{\beta}},
\quad n \geq 0,
\end{equation*}
where
\begin{equation*}
\beta = \frac{p(q-1)}{p-q}, \quad
C = \left( \frac{pq}{(p-q) \tau} \right)^{\frac{p(q-1)}{p-q}} \left( \frac{p}{\mu_{K_0}} \right)^{\frac{q}{p-q}} C_{K_0}^{\frac{p}{p-q}}.
\end{equation*}
\end{enumerate}
\end{theorem}

\begin{remark}
\label{Rem:coercive}
By setting $\mathcal{N} = \{ 0 \}$ in \cref{Thm:conv,Thm:conv_sharp}, we recover the convergence results for the parallel subspace correction method for coercive problems~\cite{Park:2020,PX:2024}.
\end{remark}

\section{Convergence analysis for nearly semicoercive problems}
\label{Sec:Nearly}
In~\cite{LWXZ:2007,WZ:2014}, it was proven that subspace correction methods for nearly singular linear systems on Hilbert spaces achieve convergence rate estimates that are independent of the nearly singular behavior of the systems, assuming appropriate conditions on the space decompositions.
Namely, if the near null space can be decomposed into a sum of local near null spaces, parameter-independent estimates for the convergence rates can be established~\cite[Theorems~4.2 and~4.3]{LWXZ:2007}.
In this section, we extend these results to nearly semicoercive convex optimization problems in Banach spaces, thereby enabling applications to a broader class of nonlinear problems.

\subsection{Orthogonal decompositions of Banach spaces}
A key tool used in the convergence analysis of nearly singular linear systems on Hilbert spaces in~\cite{LWXZ:2007,WZ:2014} is the fact that the kernel of the singular part admits an orthogonal complement; see~\cite[Lemma~4.5]{LWXZ:2007}.
Unfortunately, in general Banach spaces, not every closed subspace has a complement~\cite{LT:1971}, and moreover, there is no inherent orthogonality structure.
Therefore, to analyze nearly coercive problems in Banach spaces, we must employ alternative tools.

For this purpose, we present decomposition results for Banach spaces based on a generalized notion of orthogonality, introduced in~\cite{Alber:2000,Alber:2005}.
To achieve this, we require a stronger assumption on Banach spaces than just reflexivity.
Specifically, let $V$ be a uniformly convex and uniformly smooth Banach space, which ensures that $V$ is reflexive~\cite[Theorems~II.2.9 and~II.2.15]{Cioranescu:1990}.
It is also worth noting that many Sobolev spaces associated with nonlinear PDEs are uniformly convex and uniformly smooth~\cite{Adams:1975}.

The normalized duality mapping $J \colon V \rightrightarrows V^*$ on $V$ is defined as
\begin{equation*}
J (v) = \{ v^* \in V^* : \langle v^*, v \rangle = \| v^* \| \| v \|, \text{ } \|v^* \| = \|v\| \},
\quad v \in V.
\end{equation*}
Note that the uniform convexity and uniform smoothness of $V$ ensure that $J$ is bijective~\cite[Proposition~II.3.6]{Cioranescu:1990}, allowing us to identify $J$ as a mapping $J \colon V \rightarrow V^*$.
Similarly, we define $J^* \colon V^* \rightrightarrows V$ as the normalized duality mapping on $V^*$:
\begin{equation*}
    J^* (v^*) = \{ v \in V : \langle v^*, v \rangle = \| v^* \| \| v \|, \text{ } \| v \| = \| v^* \| \},
    \quad v^* \in V^*.
\end{equation*}
Then, we have $J^* = J^{-1}$.

Given a closed subspace $\mathcal{M}$ of $V$, we denote its polar set as $\mathcal{M}^{\circ}$:
\begin{equation*}
\mathcal{M}^{\circ} = \{ v^* \in V^* : \langle v^*, v \rangle = 0 \text{ for all } v \in \mathcal{M} \}.
\end{equation*}
The following proposition provides an orthogonal decomposition of $V$ into $\mathcal{M}$ and its complement, which is a nonlinear manifold in general~\cite{Alber:2005}, involving $\mathcal{M}^{\circ}$ and the normalized duality mapping $J^*$.
While this result appeared previously in~\cite[Theorem~2.13]{Alber:2005}, we include the proof of this result for the sake of completeness.

\begin{proposition}
\label{Prop:orthogonal}
Let $V$ be a uniformly convex and uniformly smooth Banach space, and let $\mathcal{M}$ be a closed subspace of $V$.
Then, $V$ admits a decomposition $V = \mathcal{M} + J^* \mathcal{M}^{\circ}$, which satisfies the following:
\begin{enumerate}[(i)]
\item Each element $v \in V$ has a unique decomposition $v = \phi + \xi$ with $\phi \in \mathcal{M}$ and $\xi \in J^* \mathcal{M}^{\circ}$.
Moreover, we have
\begin{equation}
\label{Prop1:orthogonal}
\phi = \operatornamewithlimits{\arg\min}_{w \in \mathcal{M}} \|v - w \|.
\end{equation}
\item $\langle J \xi, \phi \rangle = 0$ for any $\phi \in \mathcal{M}$ and $\xi \in J^* \mathcal{M}^{\circ}$.
\item $\mathcal{M} \cap J^* \mathcal{M}^{\circ} = \{ 0 \}$.
\end{enumerate}
\end{proposition}
\begin{proof}
The validity of~(ii) follows directly from the definition of the polar set $\mathcal{M}^{\circ}$.
To prove~(iii), let $v \in \mathcal{M} \cap J^* \mathcal{M}^{\circ}$.
Since $v \in \mathcal{M}$ and $J v = (J^*)^{-1}v \in \mathcal{M}^{\circ}$, we have $\langle J v, v \rangle = 0$.
This implies $v = 0$ because of the strict monotonicity of $J$~\cite[Theorem~II.1.8]{Cioranescu:1990}.

Finally, we prove~(i) using an argument similar as in~\cite{Alber:2000}.
Take any $v \in V$, and we define $\phi \in \mathcal{M}$ as in~\eqref{Prop1:orthogonal}.
The optimality condition of $\phi$ reads as
\begin{equation*}
\langle J (v - \phi), w \rangle = 0
\quad \forall w \in \mathcal{M},
\end{equation*}
which is equivalent to $J (v - \phi) \in \mathcal{M}^{\circ}$.
Thus, we have $\xi := v - \phi \in J^* \mathcal{M}^{\circ}$, leading to the desired decomposition $v = \phi + \xi \in \mathcal{M} + J^* \mathcal{M}^{\circ}$.
The uniqueness of this decomposition follows directly from~(iii).
\end{proof}

As a direct consequence of \cref{Prop:orthogonal}, we obtain the following corollary, which will play an important role in the convergence analysis of nearly semicoercive problems.

\begin{corollary}
\label{Cor:orthogonal}
Let $V$ be a uniformly convex and uniformly smooth Banach space, and let $\mathcal{M}$ be a closed subspace of $V$.
Then we have the following:
\begin{enumerate}[(a)]
\item For any $q \geq 1$, there exists a positive constant $C_q$, depending only on $q$, such that
\begin{equation}
\label{C_q}
\| \phi \|^q + \| \xi \|^q \leq C_q \| \phi + \xi \|^q,
\quad \phi \in \mathcal{M}, \text{ } \xi \in J^* \mathcal{M}^{\circ}.
\end{equation}
\item For any $\xi \in J^* \mathcal{M}^{\circ}$, we have
\begin{equation*}
\| \xi \| = \min_{w \in \mathcal{M}} \| \xi + w \|.
\end{equation*}
\end{enumerate}
\end{corollary}
\begin{proof}
Take any $\phi \in \mathcal{M}$ and $\xi \in J^* \mathcal{M}^{\circ}$.
By \cref{Prop:orthogonal}, we have $\phi = \operatornamewithlimits{\arg\min}_{w \in \mathcal{M}} \| \phi + \xi - w \|$, which implies $ \| \xi \| \leq \| \phi + \xi \|$.
Since $\phi$ is arbitrary, we deduce that~(b) holds.
To show~(a), we observe that $\| \xi \| \leq \| \phi + \xi \|$ and that $\| \phi \| \leq \| \phi + \xi \|  + \| \xi \| \leq 2 \| \phi + \xi \|$.
Hence, we get $\| \phi \|^q + \| \xi \|^q \leq (2^q + 1 ) \| \phi + \xi \|^q$, which completes the proof.
\end{proof}

\begin{remark}
\label{Rem:Hilbert}
If $V$ is a Hilbert space, then \cref{Prop:orthogonal} holds naturally when we identify the topological dual space $V^*$ with $V$, and the duality pairing with the inner product.
Moreover, in this setting, we have the Pythagorean property, which is a special case of \cref{Cor:orthogonal}; see~\cite[Corollary~5.4]{Brezis:2011}.
\end{remark}

\subsection{Parameter-independent estimates}
Now, we consider the model nearly semicoercive convex optimization of the form
\begin{equation}
\label{model_nearly}
\min_{v \in V} \left\{ F(v) := F_0 (v) + \epsilon F_1 (v) \right\},
\end{equation}
where $V$ is a uniformly convex and uniformly smooth Banach space, $F_0 \colon V \rightarrow \mathbb{R}$ and $F_1 \colon V \rightarrow \mathbb{R}$ are G\^{a}teaux differentiable and convex functionals, and $\epsilon > 0$.
We further assume that $F_0$ is semicoercive with respect to a seminorm $| \cdot |$ in the sense of \cref{Prop:semicoercive}, and that $F_1$ is coercive.

In subspace correction methods for solving~\eqref{model_nearly} under the space decomposition~\eqref{space_decomp}, the local energy functionals $F_j$, $j \in [N]$, must be specified.
We assume that each $F_j$ is given by
\begin{equation*}
F_j (w_j; v) = F_{0,j} (w_j; v) + \epsilon F_{1,j} (w_j; v),
\quad v \in V, \text{ } w_j \in V_j,
\end{equation*}
where $F_{0,j}$ and $F_{1,j}$ are specified in the case of exact local problems as follows:
\begin{equation}
\label{local_exact_nearly}
\begin{aligned}
F_{0,j} (w_j; v) &= F_0 (v + w_j), \\
F_{1,j} (w_j; v) &= F_1 (v + w_j),
\end{aligned}
\quad v \in V, \text{ } w_j \in V_j.
\end{equation}
The general case of inexact local problems~\cite{Park:2020,PX:2024} will be considered in \cref{App:Local}.

Similar as in~\eqref{d_j}, we define
\begin{equation*}
\begin{aligned}
d_{0,j} (w_j; v) &= F_{0,j} (w_j; v) - F_0 (v) - \langle F_0' (v), w_j \rangle, \\
d_{1,j} (w_j; v) &= F_{1,j} (w_j; v) - F_1 (v) - \langle F_1' (v), w_j \rangle,
\end{aligned}
\quad v \in V, \text{ } w_j \in V_j.
\end{equation*}
Then we have
\begin{equation*}
d_j (w_j; v) = d_{0,j} (w_j; v) + \epsilon d_{1,j} (w_j; v),
\quad v \in V, \text{ } w_j \in V_j,
\end{equation*}
where $d_j(w_j; v)$ was given in~\eqref{d_j}.

To guarantee the convergence of the parallel subspace correction method for solving~\eqref{model_nearly}, we require a smoothness assumption on each $d_{0,j} (w_j; v)$ and $d_{1,j} (w_j; v)$, analogous to \cref{Ass:smooth}.
This assumption is formally stated in \cref{Ass:smooth_nearly}.

\begin{assumption}[local smoothness]
\label{Ass:smooth_nearly}
For some $q > 1$, each $d_{0,j} (w_j; v)$ and $d_{1,j} (w_j; v)$, $j \in [N]$, satisfy the following:
for any $\| \cdot \|$-bounded convex subset $K \subset V$ and $\| \cdot \|$-bounded convex subset $K_j \subset V_j$ satisfying $0 \in K_j$, we have
\begin{equation*}
    \sup_{v \in K, w_j \in K_j} \frac{d_{0,j} (w_j; v)}{\| w_j \|^q} < \infty
    \quad \text{and} \quad
    \sup_{v \in K, w_j \in K_j} \frac{d_{1,j} (w_j; v)}{\| w_j \|^q} < \infty.
\end{equation*}
\end{assumption}

For $q > 1$, we define an $\epsilon$-dependent norm $\| \cdot \|_{\epsilon, q}$ on $V$ as
\begin{equation*}
\| v \|_{\epsilon, q} := (| v |^q + \epsilon \| v \|^q)^{\frac{1}{q}},
\quad v \in V.
\end{equation*}
It is straightforward to verify that the norm $\| \cdot \|_{\epsilon, q}$ is equivalent to the original norm $\| \cdot \|$.

Under the smoothness assumption stated in \cref{Ass:smooth_nearly}, \cref{Thm:conv,Thm:conv_sharp}~(see also~\cite{Park:2020,Park:2022a,PX:2024}) indicate that a key factor determining the convergence rate of the parallel subspace correction method~(\cref{Alg:PSC}) for solving~\eqref{model_nearly} is the following constant:
\begin{equation}
\label{C_K_nearly}
C_{K_0} := q \sup_{v, v+w \in K_0} \inf_{\sum_{j=1}^N w_j = w} \frac{\sum_{j=1}^N d_{j} (w_j; v)}{\| w \|_{\epsilon, q}^q},
\end{equation}
where the set $K_0$ was given in~\eqref{K_0}, and each $w_j$ belongs to $V_j$.

\begin{remark}
\label{Rem:K_0_epsilon}
Since the energy functional $F$ in~\eqref{model_nearly} depends on $\epsilon$, the set $K_0$ defined in~\eqref{K_0} also implicitly depends on $\epsilon$. 
Throughout this paper, by an $\epsilon$-independent estimate, we mean an estimate that is independent of $\epsilon$ except for its potential dependence on $K_0$.
We remark that in many applications, such as linear problems, the estimates involving $K_0$ presented in this paper hold uniformly over all bounded convex subsets $K \subset V$.
\end{remark}

In order to derive an $\epsilon$-independent upper bound for $C_{K_0}$, we need to impose additional assumptions on the space decomposition and the local problems.
The first of these assumptions, summarized in \cref{Ass:kernel}, requires that the kernel of the semicoercive functional $F_0$ in~\eqref{model_nearly} can be decomposed into a sum of local kernels~(cf.~\cite[equation~(A1)]{LWXZ:2007}).

\begin{assumption}[kernel decomposition]
\label{Ass:kernel}
The kernel $\mathcal{N} = \ker F_0$ of the semicoercive functional $F_0$ in~\eqref{model_nearly}  admits a decomposition $\mathcal{N} = \sum_{j=1}^N (V_j \cap \mathcal{N})$.
\end{assumption}

The second assumption, summarized in \cref{Ass:triangle}, is that the functional $d_{1,j} (\cdot; v)$ satisfies a triangle inequality-like property.

\begin{assumption}[triangle inequality-like property]
\label{Ass:triangle}
For any bounded convex subset $K \subset V$, there exists a positive constant $C_{K, \mathrm{tri}}$ such that
\begin{equation}
\label{C_tri}
d_{1,j} (v_j + w_j; v) \leq C_{K, \mathrm{tri}} \left( d_{1,j} (v_j; v) + d_{1,j} (w_j; v) \right),
\quad v \in K, \text{ } v_j,w_j \in V_j, \text{ } j \in [N].
\end{equation}
\end{assumption}

In the case of exact local problems~\eqref{local_exact_nearly}, a straightforward sufficient condition for \cref{Ass:triangle} is that the Bregman distance $d_{F_1}$ associated with the functional $F_1$~(cf.~\eqref{d_F}) satisfies the following triangle-inequality-like property:
\begin{equation*}
d_{F_1} (w_1 + w_2; v) \leq C_{K, \mathrm{tri}} \left( d_{F_1} (w_1; v) + d_{F_1} (w_2; v) \right),
\quad v \in K, \text{ } w_1, w_2 \in V,
\end{equation*}
for some positive constant $C_{K, \mathrm{tri}}$.
This property is indeed satisfied by a broad class of convex functionals, making it a practical and verifiable condition.
We provide a detailed discussion of this property in \cref{App:Triangle}.

Now, we are ready to present the main result of this section, \cref{Thm:nearly}, which provides an $\epsilon$-independent convergence rate estimate of the parallel subspace correction method for solving~\eqref{model_nearly}.
We note that \cref{Thm:nearly} generalizes the existing result~\cite[Theorem~3.1]{WZ:2014} for nearly singular linear problems to the context of nearly semicoercive convex optimization problems.

\begin{theorem}
\label{Thm:nearly}
Let $V$ be a uniformly convex and uniformly smooth Banach space.
Suppose that \cref{Ass:kernel,Ass:triangle,Ass:smooth_nearly} hold.
Then the constant $C_{K_0}$ given in~\eqref{C_K_nearly} has an upper bound independent of $\epsilon$.
More precisely, we have
\begin{multline*}
C_{K_0} \leq q C_q \sup_{\substack{v \in K_0, \phi \in \mathcal{N}, \xi \in J^* \mathcal{N}^{\circ}, \\ v+ \phi+\xi \in K_0}} \Bigg[ C_q C_{K_0, \mathrm{tri}} \inf_{\sum_{j=1}^N \phi_j = \phi} \frac{\sum_{j=1}^N d_{1,j} (\phi_j; v)}{\| \phi \|^q} \\
+ \inf_{\sum_{j=1}^N \xi_j = \xi}  \left( \frac{\sum_{j=1}^N d_{0,j} (\xi_j; v)}{| \xi |^q} + C_q C_{K_0, \mathrm{tri}} \frac{\sum_{j=1}^N d_{1,j} (\xi_j; v)}{\| \xi \|^q } \right) \Bigg] < \infty,
\end{multline*}
where $C_{q}$ and $C_{K_0, \mathrm{tri}}$ were given in~\eqref{C_q} and~\eqref{C_tri}, respectively, and each $\phi_j$ and $\xi_j$ are taken from $V_j \cap \mathcal{N}$ and $V_j$, respectively.
\end{theorem}
\begin{proof}
Invoking \cref{Prop:orthogonal}, we have
\begin{equation}
\label{Thm1:nearly}
C_{K_0} = q \sup_{\substack{v \in K_0, \phi \in \mathcal{N}, \xi \in J^* \mathcal{N}^{\circ}, \\ v+ \phi+\xi \in K_0}} \inf_{\sum_{j=1}^N \phi_j = \phi, \sum_{j=1}^N \xi_j = \xi}  \frac{\sum_{j=1}^N d_j (\phi_j + \xi_j; v)}{\| \phi + \xi \|_{\epsilon, q}^q}.
\end{equation}
To estimate the right-hand side of~\eqref{Thm1:nearly}, we choose any $v \in K_0$, $\phi \in \mathcal{N}$, and $\xi \in J^* \mathcal{N}^{\circ}$ such that $v + \phi + \xi \in K_0$.
For simplicity, we assume that $\phi \neq 0$ and $\xi \neq 0$; the cases where either $\phi = 0$ or $\xi = 0$ are straightforward.
Thanks to~\eqref{space_decomp} and~\cref{Ass:kernel}, we can decompose $\phi$ and $\xi$ as $\phi = \sum_{j=1}^N \phi_j$ and $\xi = \sum_{j=1}^N \xi_j$, where each $\phi_j$ belongs to $V_j \cap \mathcal{N}$ and each $\xi_j$ belongs to $V_j$.

We first deduce an upper bound for $\sum_{j=1}^N d_j (\phi_j + \xi_j; v)$~(cf.~\cite[Lemma~4.5]{LWXZ:2007}):
\begin{equation}
\label{Thm2:nearly}
\begin{split}
\sum_{j=1}^N d_j (\phi_j + \xi_j; v)
&= \sum_{j=1}^N d_{0,j} (\xi_j; v) + \epsilon \sum_{j=1}^N d_{1,j} (\phi_j + \xi_j; v) \\
&\leq \sum_{j=1}^N d_{0,j} (\xi_j ; v) + C_{K_0, \mathrm{tri}} \epsilon \sum_{j=1}^N \left( d_{1,j} (\phi_j; v) + d_{1,j} (\xi_j; v) \right),
\end{split}
\end{equation}
where the inequality is due to \cref{Ass:triangle}.
Next, by invoking \cref{Cor:orthogonal}(a), we obtain a lower bound for $\| \phi + \xi \|_{\epsilon, q}^q$ as follows:
\begin{equation}
\label{Thm3:nearly}
\| \phi + \xi \|_{\epsilon, q}^q
= | \xi |^q + \epsilon \| \phi + \xi \|^q
\geq | \xi |^q + C_q^{-1} \epsilon \left( \| \phi \|^q + \| \xi \|^q \right).
\end{equation}
Combining~\eqref{Thm2:nearly} and~\eqref{Thm3:nearly} yields
\begin{equation}
\label{Thm4:nearly}
\frac{\sum_{j=1}^N d_j (\phi_j + \xi_j; v)}{\| \phi + \xi \|_{\epsilon, q}^q}
\leq \frac{\sum_{j=1}^N d_{0,j} (\xi_j; v)}{ | \xi |^q}
+ C_q C_{K_0, \mathrm{tri}} \sum_{j=1}^N \left( \frac{d_{1,j} (\phi_j; v)}{\| \phi \|^q} + \frac{d_{1,j} (\xi_j; w)}{\| \xi \|^q} \right).
\end{equation}
Since the decompositions $\phi = \sum_{j=1}^N \phi_j$ and $\xi = \sum_{j=1}^N \xi_j$ were arbitrarily chosen, by invoking~\eqref{Thm1:nearly} and~\eqref{Thm4:nearly}, we obtain
\begin{multline}
\label{Thm5:nearly}
C_{K_0} \leq q \sup_{\substack{v \in K_0, \phi \in \mathcal{N}, \xi \in J^* \mathcal{N}^{\circ}, \\ v+ \phi+\xi \in K_0}} \Bigg[ C_q C_{K_0, \mathrm{tri}} \inf_{\sum_{j=1}^N \phi_j = \phi} \frac{\sum_{j=1}^N d_{1,j} (\phi_j; v)}{\| \phi \|^q} \\ + \inf_{\sum_{j=1}^N \xi_j = \xi}  \left( \frac{\sum_{j=1}^N d_{0,j} (\xi_j; v)}{| \xi |^q} + C_q C_{K_0, \mathrm{tri}} \frac{\sum_{j=1}^N d_{1,j} (\xi_j; v)}{\| \xi \|^q } \right) \Bigg].
\end{multline}

It remains to show that the right-hand side of~\eqref{Thm5:nearly} is finite.
We write
\begin{equation*}
M_{K_0} := \sup_{v \in K_0} \| v \| < \infty.
\end{equation*}
Then, for any $v \in K_0$, $\phi \in \mathcal{N}$, and $\xi \in J^* \mathcal{N}^{\circ}$ such that $v + \phi + \xi \in K_0$, we have
\begin{equation*}
\| \xi \| \stackrel{\text{(i)}}{\leq} \| \phi + \xi \|
\leq  \| v \| + \| v + \phi + \xi \| \leq 2 M_{K_0}
\end{equation*}
and
\begin{equation*}
\| \phi \| \leq \| v + \phi + \xi \| + \| v \| + \| \xi \| \leq 3 M_{K_0},
\end{equation*}
where (i) is due to \cref{Cor:orthogonal}(b).
In addition, thanks to~\eqref{Lem1:seminorm} and \cref{Cor:orthogonal}(b), we have
\begin{equation*}
| \xi |^q \geq C \inf_{w \in \mathcal{N}} \| \xi + w \|^q = C \| \xi \|^q.
\end{equation*}
for some positive constant $C$.
With these bounds, the finiteness of the right-hand side of~\eqref{Thm5:nearly} follows directly from a similar argument as in the proof of \cref{Lem:stable}.
\end{proof}

As presented in \cref{Thm:conv_sharp}, an enhanced convergence rate estimate for the parallel subspace correction method can be achieved if we have an additional assumption that the energy functional is sharp.
In the following, we present a relevant result for nearly semicoercive problems.
First, we present in \cref{Ass:uniform}, which introduces an appropriate sharpness assumption for the nearly semicoercive problem~\eqref{model_nearly}, where $d_{F_0}$ and $d_{F_1}$ are defined in the same manner as in~\eqref{d_F}.
It is readily observed that \cref{Ass:uniform} provides a sufficient condition to guarantee that the energy functional $F$ in~\eqref{model_nearly} satisfies \cref{Ass:sharp}.

\begin{assumption}[uniform convexity]
\label{Ass:uniform}
For some $p > 1$, we have the following:
\begin{enumerate}[(a)]
\item For any $| \cdot |$-bounded convex subset $K$ of $V$, we have
\begin{equation*}
    \mu_{0,K} := p \inf_{v, v+w \in K} \frac{d_{F_0} (w; v)}{| w |^p} > 0.
\end{equation*}
\item For any $\| \cdot \|$-bounded convex subset $K$ of $V$, we have
\begin{equation*}
    \mu_{1,K} := p \inf_{v, v+w \in K} \frac{d_{F_0} (w; v) + d_{F_1} (w; v)}{\| w \|^p} > 0.
\end{equation*}
\end{enumerate} 
\end{assumption}

Under \cref{Ass:smooth_nearly,Ass:uniform}, \cref{Thm:conv_sharp} indicates that the constant $\mu_{K_0}$ given below is a critical factor in determining the convergence rate of the subspace correction method, alongside $C_{K_0}$:
\begin{equation}
\label{mu_K_nearly}
\mu_{K_0} := p \inf_{v \in K_0} \frac{F(v) - F(u)}{\| v - u \|_{\epsilon, q}^p},
\end{equation}
where the set $K_0$ was given in~\eqref{K_0}.
The following theorem states that $\mu_{K_0}$ has a lower bound independent of $\epsilon$ if $\epsilon$ is small enough.

\begin{theorem}
\label{Thm:nearly_sharp}
Let $V$ be a uniformly convex and uniformly smooth Banach space.
Suppose that \cref{Ass:smooth_nearly,Ass:uniform} hold with $p \geq q$.
If $\epsilon \in (0, 1/2]$, then the constant $\mu_{K_0}$ given in~\eqref{mu_K_nearly} has a lower bound independent of $\epsilon$.
\end{theorem}
\begin{proof}
Take any $v \in K_0 \setminus \{ u \}$, and let $w = v - u$.
We derive an upper bound for $\| w \|_{\epsilon, q}^p$ as follows:
\begin{equation}
\label{Prop1:sharp}
\| w \|_{\epsilon, q}^p = (|w|^q + \epsilon \| w \|^q )^{\frac{p}{q}}
\leq 2^{\frac{p}{q} - 1} (|w|^p + \epsilon^{\frac{p}{q}} \| w \|^p)
\leq 2^{\frac{p}{q} - 1} (|w|^p + \epsilon \|w \|^p),
\end{equation}
where the first inequality follows from the elementary inequality
\begin{equation*}
(a + b)^s \leq 2^{s-1} (a^s + b^s),
\quad a,b \geq 0, \text{ } s \geq 1,
\end{equation*}
and the second inequality is because of $\epsilon < 1$.
Note that the set $K_0$ is both $| \cdot |$-bounded and $\| \cdot \|$-bounded.
Hence, it follows that
\begin{multline*}
\mu_{K_0} \stackrel{\eqref{Prop1:sharp}}{\geq} \frac{d_{F_0} (w; u) + \epsilon d_{F_1} (w; u)}{2^{\frac{p}{q}-1} ( |w |^p + \epsilon \|w \|^p)}
\stackrel{\text{(i)}}{\geq} \frac{d_{F_0} (w; u) + \epsilon (d_{F_0}(w; u) + d_{F_1} (w; u))}{2^{\frac{p}{q}} ( |w |^p + \epsilon \|w \|^p)} \\
\stackrel{\text{(ii)}}{\geq} \frac{\mu_{0,K_0} |w|^p + \epsilon \mu_{1,K_0} \| w \|^p}{p 2^{\frac{p}{q}} ( |w |^p + \epsilon \|w \|^p)}
\geq \frac{\min \{ \mu_{0,K_0}, \mu_{1,K_0} \}}{p 2^{\frac{p}{q}}},
\end{multline*}
which completes the proof, where (i) is because of $1 - \epsilon \geq 1/2$ and (ii) is due to \cref{Ass:uniform}.
\end{proof}




\section{Applications}
\label{Sec:Applications}
In this section, we present several applications of the proposed convergence theory for subspace correction methods.
We analyze the convergence of two-level domain decomposition methods for solving some nonlinear PDEs with associated energy functionals that are either semicoercive or nearly semicoercive.

Let $\Omega$ be a bounded polyhedral domain in $\mathbb{R}^d$, and let $\mathcal{T}_h$ be a quasi-uniform triangulation of $\Omega$ with $h$ the characteristic element diameter.
We denote by $S_h (\Omega)$ the continuous and piecewise linear finite element space defined on $\mathcal{T}_h$:
\begin{equation*}
S_h (\Omega) = \left\{ v \in C (\Omega) : v |_T \in \mathbb{P}_1 (T) \text{ for all } T \in \mathcal{T}_h \right\}.
\end{equation*}

Assume that we also have a quasi-uniform triangulation $\mathcal{T}_H$, with $\mathcal{T}_h$ a refinement of $\mathcal{T}_H$.
We define the coarse finite element space $\mathcal{S}_H (\Omega)$ similarly to $S_h (\Omega)$:
\begin{equation*}
S_H (\Omega) = \left\{ v \in C (\Omega) : v |_T \in \mathbb{P}_1 (T) \text{ for all } T \in \mathcal{T}_H \right\}.
\end{equation*}
Since $S_H (\Omega) \subset S_h (\Omega)$, the natural embedding operator $I_0 \colon S_H (\Omega) \rightarrow S_h (\Omega)$ is well-defined.

Let $\{ \Omega_j \}_{j=1}^N$ be a quasi-uniform overlapping domain decomposition of $\Omega$, where each subdomain $\Omega_j$ is a union of $\mathcal{T}_h$-elements and has diameter of order $H$.
The overlap width among the subdomains is measured by a parameter $\delta$.
For each $j \in [N]$, we define the local finite element space $S_h (\Omega_j)$ as follows:
\begin{equation*}
S_h (\Omega_j) = \left\{ v \in C (\Omega_j) : v|_T \in \mathbb{P}_1 (T) \text{ for all } T \in \mathcal{T}_h |_{\Omega_j} ,
\text{ } v = 0 \text{ on } \partial \Omega_j \setminus \partial \Omega \right\}.
\end{equation*}
The operator $I_j \colon S_h (\Omega_j) \rightarrow S_h (\Omega)$ is then defined as the extension-by-zero operator.

In what follows, the notation $A \lesssim B$ means that there exists a constant $c > 0$ independent of $h$, $H$, and $\delta$, such that $A \leq c B$.

\subsection{A semicoercive problem}
The $s$-Laplacian equation is an important nonlinear PDE; see~\cite{Loisel:2020}, where an efficient numerical method was proposed, as well as the references therein.
In this paper, we consider the corresponding Neumann boundary value problem:
\begin{align*}
- \nabla \cdot \left( |\nabla u|^{s-2} \nabla u \right) = f \quad &\textrm{ in } \Omega, \\
\frac{\partial u}{\partial \nu} = 0 \quad &\textrm{ on } \partial \Omega,
\end{align*}
where $s > 1$, $f \in (W^{1,s}(\Omega))^*$, and $\nu$ is the unit outer normal to $\partial \Omega$.
For this problem to be solvable, it is required that $f$ satisfies the compatibility condition $\langle f, 1 \rangle = 0$.
This equation is well-known to have the following variational formulation:
\begin{equation*}
\min_{v \in W^{1,s} (\Omega)} \left\{ \frac{1}{s} \int_{\Omega} |\nabla v|^s \,dx - \left< f, v \right> \right\}.
\end{equation*}
To solve this variational problem numerically, we consider the following finite element discretization defined on $S_h (\Omega)$:
\begin{equation}
\label{sLap_FEM}
\min_{v \in S_h (\Omega)} \left\{ \frac{1}{s} \int_{\Omega} |\nabla v|^s \,dx - \left< f, v \right> \right\}.
\end{equation}
Error estimates for~\eqref{sLap_FEM} can be found in, e.g.,~\cite{BL:1993}.
We observe that~\eqref{sLap_FEM} is a specific instance of~\eqref{model}.
Namely, we obtain~\eqref{sLap_FEM}  by setting
\begin{equation*}
V = S_h (\Omega),
\quad
F(v) = \frac{1}{s} \int_{\Omega} | \nabla v |^s \,dx - \langle f, v \rangle.
\end{equation*}
in~\eqref{model}.
It is straightforward to verify that $F$ is semicoercive with respect to the $W^{1,s} (\Omega)$-seminorm, whose kernel is $\mathcal{N} = \operatorname{span} \{ 1 \}$.

In the following, we analyze a two-level additive Schwarz method for solving the Neumann boundary value problem~\eqref{sLap_FEM}.
More precisely, we prove that the two-level additive Schwarz method is scalable in the sense that the dependence of its convergence rate on the geometric parameters $h$, $H$, and $\delta$ is only through the ratio $H / \delta$~\cite{TW:2005}.
Note that the case of the Dirichlet boundary condition was considered in several existing works, e.g.,~\cite{LP:2024,TX:2002}.
We define the subspaces $\{ V_j \}_{j=0}^N$ of $V = S_h (\Omega)$ as follows:
\begin{equation}
\label{two_level}
V_0 = I_0 S_H (\Omega),
\quad V_j = I_j S_h (\Omega_j),
\quad j \in [N],
\end{equation}
so that we have the two-level space decomposition
\begin{equation*}
V = V_0 + \sum_{j=1}^N V_j.
\end{equation*}
If we employ this two-level space decomposition and the exact local problems~\eqref{local_exact} in the parallel subspace correction method presented in \cref{Alg:PSC}, then we obtain the two-level additive Schwarz method.

Thanks to \cref{Thm:conv_sharp}, it suffices to verify \cref{Ass:smooth,Ass:sharp} and estimate the constants $\tau_0$, $\mu_{K_0}$, and $C_{K_0}$ given in~\eqref{tau_0},~\eqref{K_0}, and~\eqref{C_K}, respectively, to estimate the convergence rate of \cref{Alg:PSC}.
The strengthened convexity parameter $\tau_0$ has a lower bound $\tau_0 \geq \frac{1}{5}$, due to a usual coloring argument~\cite[Section~5.1]{Park:2020}.
Proceeding as in~\cite[Section~6.1]{Park:2020}, we can verify that \cref{Ass:smooth,Ass:sharp} hold with $p = \max \{s, 2\}$, $q = \min \{s, 2 \}$, and $\mu_{K_0} \gtrsim 1$~(cf.~\cite[Equations~(6.6) and~(6.7)]{Park:2020}).

To estimate $C_{K_0}$, we examine the dependence of constants appearing in the semicoercive analysis in \cref{Sec:Semicoercive} on $H/\delta$.
The Poincar\'{e} inequality implies that~\eqref{Lem1:seminorm} can be written as
\begin{equation}
\label{sLap_analysis1}
\inf_{\phi \in \mathcal{N}} \| v + \phi \|_{W^{1,s} (\Omega)} \lesssim  | v |_{W^{1,s} (\Omega)},
\quad v \in V.
\end{equation}
Moreover, by~\cite[Lemma~4.1]{TX:2002},~\eqref{norm_stable} takes the form
\begin{equation}
\label{sLap_analysis2}
\sup_{\| w \|_{W^{1,s} (\Omega)} = 1} \inf_{w = \sum_{j=1}^N w_j} \left( \sum_{j=1}^N \| w_j \|_{W^{1,s} (\Omega)}^q \right)^{\frac{1}{q}} \lesssim C_{H/\delta},
\end{equation}
where $C_{H/\delta}$ denotes a generic constant that depends on the geometric parameters only through $H/\delta$.
Combining~\eqref{sLap_analysis1} and~\eqref{sLap_analysis2}, and following the proof of \cref{Lem:seminorm_stable}, we obtain
\begin{equation}
\label{sLap_analysis3}
\sup_{| w |_{W^{1,s} (\Omega)} = 1} \inf_{\phi \in \mathcal{N}} \inf_{\sum_{j=1}^N w_j = w + \phi} \left( \sum_{j=1}^N \| w_j \|^q \right)^{\frac{1}{q}} \lesssim C_{H/\delta}.
\end{equation}
In the proof of \cref{Lem:stable}, \eqref{sLap_analysis3} implies that the constant $C$ in~\eqref{Lem2:stable}, and hence each $L_j$ in~\eqref{Lem3:stable}, depends on the geometry only through $H/\delta$.
Consequently, \eqref{Lem4:stable} implies that $C_{K_0}$ also admits an upper bound whose geometric dependence is only on $H/\delta$.
In conclusion, by \cref{Thm:conv_sharp}, the two-level additive Schwarz method for solving~\eqref{sLap_FEM} satisfies the following convergence estimate:
\begin{equation*}
F(u^{(n)}) - F(u) \lesssim \frac{C_{H/\delta}}{n^{\frac{p(q-1)}{p - q}}}, \quad n \geq 0.
\end{equation*}

\begin{remark}
\label{Rem:TV}
An interesting extreme case of the problem discussed here arises when $s = 1$, corresponding to total variation minimization, which is important in mathematical imaging~\cite{LP:2020}.
In this case, the convergence theory presented in this paper is not directly applicable since the energy functional in~\eqref{sLap_FEM} becomes nondifferentiable.
Indeed, a counterexample demonstrating nonconvergence of a subspace correction method is given in~\cite{LN:2017}.
To address this issue, one approach is to regularize the problem to make it differentiable; for example, see~\cite{TWZZ:2023}, where an efficient numerical method was developed from this perspective.
Alternatively, one may consider subspace correction methods applied to a suitable dual formulation, as discussed in~\cite{CTWY:2015,LP:2019}.
\end{remark}

\subsection{A nearly semicoercive problem}
As a next example, we consider a Poisson-type equation with a nonlinear mass term, which was also considered in~\cite{CHW:2020,Park:2024a}, given by
\begin{align*}
- \Delta u + \epsilon |u|^{s-2} u = f \quad &\textrm{ in } \Omega, \\
\frac{\partial u}{\partial \nu} = 0 \quad &\textrm{ on } \partial \Omega,
\end{align*}
where $s \in [2, \infty)$ when $d = 2$ and $s \in [2, 6]$ when $d = 3$, $f \in (H^1 (\Omega))^*$, and $\epsilon > 0$.
The above equation admits the following variational formulation~\cite[Theorem~7.1]{CHW:2020}:
\begin{equation*}
\min_{v \in H^{1} (\Omega)} \left\{ \frac{1}{2} \int_{\Omega} |\nabla v|^2 \,dx + \frac{\epsilon}{s} \int_{\Omega} |v|^s \,dx - \left< f, v \right> \right\}.
\end{equation*}
The finite element discretization of this variational formulation defined on $S_h (\Omega)$ is given by
\begin{equation}
\label{perturbed_FEM}
\min_{v \in S_h (\Omega)} \left\{ \frac{1}{2} \int_{\Omega} |\nabla v|^2 \,dx + \frac{\epsilon}{s} \int_{\Omega} |v|^s \,dx - \left< f, v \right> \right\}.
\end{equation}
We observe that, if we set
\begin{equation*}
V = S_h (\Omega),
\quad F_0 (v) = \frac{1}{2} \int_{\Omega} | \nabla v |^2 \,dx - \langle f, v \rangle,
\quad F_1 (v) = \frac{1}{s} \int_{\Omega} |v|^s \,dx
\end{equation*}
in the abstract nearly semicoercive problem~\eqref{model_nearly}, then we obtain~\eqref{perturbed_FEM}.

In the following, we prove that the convergence rate of the two-level additive Schwarz method, specifically \cref{Alg:PSC} equipped with the two-level space decomposition~\eqref{two_level} and the exact local problems, for solving~\eqref{perturbed_FEM}, achieves an $\epsilon$-independent estimate when $\epsilon$ is sufficiently small.
In the framework introduced in \cref{Sec:Nearly}, we set the seminorm $| \cdot |$ and norm $\| \cdot \|$ as $| \cdot | = | \cdot |_{H^1 (\Omega)}$ and $\| \cdot \| = \| \cdot \|_{H^1 (\Omega)}$, respectively.
Using a similar argument as in the previous example and the Sobolev inequality
\begin{equation*}
\| u \|_{L^2 (\Omega)} \lesssim \| u \|_{L^s (\Omega)} \lesssim \| u \|_{H^1 (\Omega)},
\end{equation*}
together with the inequality~(cf.~\cite[Lemma~2.1]{BL:1993})
\begin{equation*}
\| w \|_{L^s (\Omega)}^s
\lesssim d_{F_1} (w; v)
\lesssim \| w \|_{L^s (\Omega)}^2,
\quad v, v+w \in K_0,
\end{equation*}
we verify that \cref{Ass:smooth_nearly,Ass:uniform} hold with $p = s$ and $q = 2$.
\cref{Ass:triangle} can be verified by a similar argument as \cref{Ex:triangle_sLap}.
Moreover, \cref{Ass:kernel} is satisfied since $\operatorname{span} \{ 1 \} \subset V_0$.
Therefore, by \cref{Thm:nearly,Thm:nearly_sharp}, we conclude that the convergence rate of the two-level additive Schwarz method for solving~\eqref{perturbed_FEM} achieves an $\epsilon$-independent estimate.

\section{Concluding remarks}
\label{Sec:Conclusion}
In this paper, we presented a convergence analysis of subspace correction methods for semicoercive and nearly semicoercive convex optimization problems, generalizing the theory of singular~\cite{LWXZ:2008,WLXZ:2008} and nearly singular~\cite{LWXZ:2007,WZ:2014} linear problems.
The central message is that the elegant theoretical results developed for linear problems can be directly extended to convex optimization problems.
Given the wide applicability of the linear theory to various PDEs~\cite{FMW:2019,LWC:2009,XY:2015}, we anticipate that the convex theory introduced in this paper will similarly have broad applications to nonlinear PDEs.

We conclude this paper by discussing potential directions for future work related to the abstract theory. 
While this paper focused on smooth convex optimization problems, an important next step is to extend the framework to constrained or nonsmooth convex optimization problems~\cite{BK:2012,BTW:2003,Carstensen:1997}.
Such extensions are particularly important due to the prevalence of these problems in scientific and engineering applications; see, for example,~\cite{BF:2024,LP:2020,TWZZ:2023}.
Although the convergence theory of subspace correction methods for nonsmooth but coercive problems has been developed in~\cite{Park:2020}, extending this theory to the semicoercive setting presents challenges.
In particular, it was shown in~\cite{CTWY:2015,LP:2019} that domain decomposition methods for dual total variation minimization, a semicoercive convex problem with pointwise constraints, achieve only sublinear convergence rates, despite the fact that the associated energy functional satisfies \cref{Ass:smooth,Ass:sharp} with $p = q = 2$.

On the other hand, a recent work~\cite{Park:2024a} demonstrated that domain decomposition methods for certain semilinear elliptic problems achieve convergence rates that are independent of the nonlinearity of the problems.
We expect that there may be a connection between the result in~\cite{Park:2024a} and the nearly semicoercive theory presented in this paper, though further investigation is needed.

\appendix
\section{Inexact local problems}
\label{App:Local}
Subspace correction methods often incorporate inexact local problems.
Namely, each local energy functional $F_j$, $j \in [N]$, in \cref{Alg:PSC} may not be defined exactly as in~\eqref{local_exact}, but rather as an approximation.
From this perspective, the convergence analyses of subspace correction methods for convex optimization presented in~\cite{Park:2020,Park:2022a,PX:2024} were conducted allowing inexact local problems.
In this appendix, we demonstrate how the convergence analysis in this paper can be extended to accommodate inexact local problems.

\subsection{Semicoercive problems}
We first consider subspace correction methods for semicoercive problems discussed in \cref{Sec:Semicoercive}.
That is, in the model problem~\eqref{model}, we assume that $F$ is semicoercive with respect to a seminorm $| \cdot |$ in the sense of \cref{Prop:semicoercive}, and denote $\mathcal{N} = \ker F = \ker | \cdot |$.
We assume that each local energy functional $F_j$, $j \in [N],$ is not necessarily given by the exact one~\eqref{local_exact}, but rather by any functional satisfying the smoothness condition stated in \cref{Ass:smooth}, where $d_j$ is still defined as in~\eqref{d_j}.
Additionally, we require the further assumptions on $F_j$ summarized in \cref{Ass:local}~(cf.~\cite[Assumption~1]{PX:2024}).

\begin{assumption}[local problems]
\label{Ass:local}
For any $j \in [N]$ and $v \in V$, the local energy functional $F_j (\cdot ; v) \colon V_j \rightarrow \mathbb{R}$ satisfies the following:
\begin{enumerate}[(a)]
\item (convexity) The functional $F_j (\cdot; v) \colon V_j \rightarrow \mathbb{R}$ is G\^{a}teaux differentiable, semicoercive, and convex.
\item (consistency) We have
\begin{equation*}
F_j (0;v) = F(v),
\end{equation*}
and
\begin{equation*}
\langle F_j' (0;v), w_j \rangle = \langle F'(v), w_j \rangle
\quad \forall w_j \in V_j.
\end{equation*}
\item (stability) For some $\omega \in (0, 1] \cup (1, \rho)$, we have
\begin{equation}
\label{omega}
d_F (w_j; v) \leq \omega d_j (w_j ; v)
\quad \forall w_j \in V_j,
\end{equation}
where the constant $\rho$ is defined as
\begin{equation}
\label{rho}
    \rho = \min_{j \in [N]} \inf_{d_j (w_j; v) \neq 0} \frac{\langle d_j' (w_j; v), w_j \rangle}{d_j (w_j; v)}.
\end{equation}
\end{enumerate}
\end{assumption}

If the local energy functional $F_j$ is given by~\eqref{local_exact}, then it clearly satisfies \cref{Ass:local}.
One can verify without difficulty that \cref{Ass:local}(a, b) implies that the constant $\rho$ defined in~\eqref{rho} satisfies $\rho \geq 1$.
For completeness, we provide an example below, as also given in~\cite[Example~2]{PX:2024}.

\begin{example}
\label{Ex:local}
Suppose that the local energy functional $F_j$ is given by
\begin{equation*}
F_j (w_j; v) = F(v) + \langle F'(v), w_j \rangle + \frac{M}{s} \| w_j \|^s,
\quad v \in V, \text{ } w_j \in V_j,
\end{equation*}
for some $s > 1$ and $M > 0$.
It is clear that \cref{Ass:local}(a, b) holds.
Moreover, for any $v \in V$ and $w_j \in V_j \setminus \{0\}$, we have~(see~\cite{XR:1991})
\begin{equation*}
\frac{\langle d_j' (w_j; v), w_j \rangle}{d_j (w_j ; v)} = s.
\end{equation*}
This implies $\rho = s$.
\end{example}

The additional assumptions for local problems presented in \cref{Ass:local} are motivated by their role in ensuring a sufficient decrease property for the local problems.
More precisely, these assumptions guarantee that solving a local problem satisfying \cref{Ass:local} leads to a reduction in the energy $F$; see \cref{Lem:local_descent}.

\begin{lemma}
\label{Lem:local_descent}
Let $V$ be a reflexive Banach space, and let $F \colon V \rightarrow \mathbb{R}$ be a G\^{a}teaux differentiable, convex, and semicoercive convex functional with the kernel $\mathcal{N}$.
For $j \in [N]$ and $v \in V$, let
\begin{equation*}
\hat{w}_j \in \operatornamewithlimits{\arg\min}_{w_j \in V_j} F_j (w_j; v).
\end{equation*}
Under \cref{Ass:local}, we have
\begin{equation*}
F(v) - F(v + \hat{w}_j) \geq \left( 1 - \frac{\omega}{\rho} \right) \langle d_j' (\hat{w}_j; v), \hat{w}_j \rangle \geq 0.
\end{equation*}
\end{lemma}
\begin{proof}
By \cref{Ass:local}(b), we have
\begin{equation*}
\hat{w}_j \in \operatornamewithlimits{\arg\min}_{w_j \in V_j} \left\{ \langle F'(v), w_j \rangle + d_j (w_j; v) \right\},
\end{equation*}
in which the optimality condition reads as
\begin{equation}
\label{hat_w_j_optimality}
\langle F'(v), w_j \rangle + \langle d_j' (\hat{w}_j; v), w_j \rangle = 0
\quad \forall w_j \in V_j.
\end{equation}
It follows that
\begin{equation*}
\begin{split}
F(v + \hat{w}_j)
&\stackrel{\eqref{omega}}{\leq} F(v) + \langle F'(v), \hat{w}_j \rangle + \omega d_j (\hat{w}_j; v) \\
&\stackrel{\eqref{rho}}{\leq} F(v) + \langle F'(v), \hat{w}_j \rangle + \frac{\omega}{\rho} \langle d_j' (\hat{w}_j; v), \hat{w}_j \rangle \\
&\stackrel{\eqref{hat_w_j_optimality}}{=} F(v) - \left(1 - \frac{\omega}{\rho} \right) \langle d_j' (\hat{w}_j; v), \hat{w}_j \rangle .
\end{split}
\end{equation*}
Finally, since $\rho \geq 1$ and $\omega \leq \rho$, we have
\begin{equation*}
\left( 1 - \frac{\omega}{\rho} \right) \langle d_j' (\hat{w}_j; v), \hat{w}_j \rangle \geq 0,
\end{equation*}
which completes the proof.
\end{proof}

Using \cref{Lem:local_descent}, we are able to derive a descent property of the parallel subspace correction method with inexact local problems, as summarized in \cref{Lem:descent_inexact}.
This result generalizes \cref{Lem:descent}.

\begin{lemma}
\label{Lem:descent_inexact}
Let $V$ be a reflexive Banach space, and let $F \colon V \rightarrow \mathbb{R}$ be a G\^{a}teaux differentiable, convex, and semicoercive convex functional with the kernel $\mathcal{N}$.
In \cref{Alg:PSC}, suppose that \cref{Ass:local} holds.
Then we have
\begin{equation*}
\resizebox{\hsize}{!}{%
$\displaystyle
F(u^{(n+1)}) \leq F(u^{(n)}) + \tau \theta \min_{w \in V} \left\{ \langle F'(u^{(n)}), w \rangle + \inf_{\phi \in \mathcal{N}} \inf_{\sum_{j=1}^N w_j = w + \phi} \sum_{j=1}^N d_j (w_j; u^{(n)}) \right\},
\text{ } n \geq 0,
$
}
\end{equation*}
where the constant $\theta$ is given by
\begin{equation*}
\theta = 
\begin{cases}
1, \quad &\text{ if } \omega \in [0, 1], \\
\frac{\rho - \omega}{\rho - 1}, \quad &\text{ if } \omega \in (1, \rho).
\end{cases}
\end{equation*}
\end{lemma}
\begin{proof}
Take any $n \geq 0$.
By~\eqref{Lem1:descent}, it suffices to estimate $\sum_{j=1}^N F(u^{(n)} + w_j^{(n+1)})$.
As the case $\omega \in (0, 1]$ can be proven with the same argument as in \cref{Lem:descent}, we focus only on the case $\omega \in (1, \rho)$.
It follows that
\begin{equation}
\label{Lem1:descent_inexact}
\begin{split}
&\sum_{j=1}^N F(u^{(n)} + w_j^{(n+1)})
\stackrel{\eqref{omega}}{\leq} N F(u^{(n)}) + \sum_{j=1}^{N} \left( \langle F'(u^{(n)}), w_j^{(n+1)} \rangle + \omega d_j (w_j^{(n+1)}; u^{(n)}) \right) \\
&= N F(u^{(n)}) + \omega \min_{w \in V} \left\{ \langle F'(u^{(n)}), w \rangle + \inf_{\phi \in \mathcal{N}} \inf_{\sum_{j=1}^N w_j = w + \phi} \sum_{j=1}^N d_j (w_j; u^{(n)}) \right\} \\
&\quad - (\omega - 1) \sum_{j=1}^N \langle F'(u^{(n)}), w_j^{(n+1)} \rangle.
\end{split}
\end{equation}
Note that \cref{Lem:local_descent} implies
\begin{equation}
\label{Lem2:descent_inexact}
\begin{split}
- \langle F'(u^{(n)}), w_j^{(n+1)} \rangle
&\stackrel{\eqref{hat_w_j_optimality}}{=} \langle d_j' (w_j^{(n+1)}; u^{(n)}), w_j^{(n+1)} \rangle \\
&\leq \frac{\rho}{\rho - \omega} \left( F(u^{(n)}) - F(u^{(n)} + w_j^{(n+1)}) \right).
\end{split}
\end{equation}
By~\eqref{Lem1:descent_inexact} and~\eqref{Lem2:descent_inexact}, we get
\begin{multline}
\label{Lem3:descent_inexact}
\sum_{j=1}^N F(u^{(n)} + w_j^{(n+1)}) \\
\leq N F(u^{(n)}) + \frac{\rho - \omega}{\rho - 1} \min_{w \in V} \left\{ \langle F'(u^{(n)}, w \rangle + \inf_{\phi \in \mathcal{N}} \inf_{\sum_{j=1}^N w_j = w + \phi} \sum_{j=1}^N d_j (w_j; u^{(n)}) \right\}.
\end{multline}
Finally, combining~\eqref{Lem1:descent} and~\eqref{Lem3:descent_inexact} completes the proof.
\end{proof}

The only distinction in the descent property for the case of inexact local problems, as described in \cref{Lem:descent_inexact}, compared to the case of exact local problems in \cref{Lem:descent}, is that $\tau$ in \cref{Lem:descent} is replaced with $\tau \theta$ in \cref{Lem:descent_inexact}.
Consequently, by following the same arguments as those in \cref{Sec:Semicoercive}, we can establish convergence theorems analogous to \cref{Thm:conv,Thm:conv_sharp} for inexact local problems.
As the statements of these results are identical to those of \cref{Thm:conv,Thm:conv_sharp}, except for replacing $\tau$ with $\tau \theta$, we omit them here for brevity.

\subsection{Nearly semicoercive problems}
Next, we consider subspace correction methods for the nearly semicoercive problems discussed in \cref{Sec:Nearly}.
As outlined in \cref{Sec:Nearly}, the analysis for nearly semicoercive problems relies on the assumptions stated in \cref{Ass:smooth_nearly,Ass:kernel,Ass:triangle,Ass:uniform}.
Among these, the assumptions specifically related to the local problems are \cref{Ass:smooth_nearly,Ass:triangle}.
Since the analysis of nearly semicoercive problems builds upon the coercive theory (using the particular norm $\| \cdot \|_{\epsilon, q}$), we can derive analogous results to \cref{Thm:nearly,Thm:nearly_sharp}, provided that additional assumptions on the local problems---such as convexity, consistency, and stability, as described in \cref{Ass:local}---are satisfied.
For the sake of brevity, we omit the detailed derivations.

\section{Proofs of the convergence theorems}
\label{App:Proofs}
In this appendix, we provide proofs of the convergence theorems of the parallel subspace correction method for semicoercive problems discussed in this paper, namely, \cref{Thm:conv,Thm:conv_sharp}.
The proofs presented in this section use similar arguments as in~\cite{Park:2020,Park:2022a}.

We begin by presenting several elementary lemmas.
We note that \cref{Lem:minimum} also appeared in~\cite[Lemma~3.8]{Park:2022a}.

\begin{lemma}
\label{Lem:minimum}
Let $a, b > 0$, $q  > 1$, and $T > 0$.
The minimum of the function $g(t) = \frac{a}{q}t^q - bt$, $t \in [0, T]$, is given as follows:
\begin{equation*}
\min_{t \in [0, T]} g(t) =
\begin{cases}
    \frac{a}{q} T^q - b T < -b T \left( 1- \frac{1}{q} \right) & \text{ if } a T^{q-1} - b < 0, \\
    - b \left(1 - \frac{1}{q} \right) \left( \frac{b}{a} \right)^{\frac{1}{q-1}} & \text{ if } a T^{q-1} - b \geq 0.
\end{cases}
\end{equation*}
\end{lemma}

The following lemma, also introduced in~\cite[Lemma~1.1]{GN:2017}, can be proven easily by invoking~\cite[Lemma~3.7]{Park:2022a}.

\begin{lemma}
\label{Lem:recurrence}
Let $\{ a_n \}$ be a sequence of positive real numbers that satisfies
\begin{equation*}
    a_n - a_{n+1} \geq C a_n^{\gamma}, \quad n \geq 0,
\end{equation*}
for some $C > 0$ and $\gamma > 1$.
Then with $\beta = \frac{1}{\gamma - 1}$, we have
\begin{equation*}
    a_n \leq \left( \frac{\beta}{C n + \beta a_0^{-1/\beta}} \right)^{\beta},
    \quad n \geq 0.
\end{equation*}
\end{lemma}

Thanks to \cref{Lem:descent}~(see \cref{Lem:descent_inexact} for the case of inexact local problems), for any $n \geq 0$, it suffices to estimate
\begin{equation*}
\min_{w \in V} \left\{ \langle F'(u^{(n)}), w \rangle + \inf_{\phi \in \mathcal{N}} \inf_{\sum_{j=1}^N w_j = w + \phi} \sum_{j=1}^N d_j (w_j; u^{(n)}) \right\}.
\end{equation*}
It follows that
\begin{equation}
\label{proof_core}
\begin{split}
\min_{w \in V} & \left\{ \langle F'(u^{(n)}), w \rangle + \inf_{\phi \in \mathcal{N}} \inf_{\sum_{j=1}^N w_j = w + \phi} \sum_{j=1}^N d_j (w_j; u^{(n)}) \right\} \\
&\leq \min_{u^{(n)} + w \in K_0} \left\{ \langle F'(u^{(n)}), w \rangle + \inf_{\phi \in \mathcal{N}} \inf_{\sum_{j=1}^N w_j = w + \phi} \sum_{j=1}^N d_j (w_j; u^{(n)}) \right\} \\
&\stackrel{\text{(i)}}{\leq} \min_{u^{(n)} + w \in K_0} \left\{ \langle F'(u^{(n)}), w \rangle + \frac{C_{K_0}}{q} | w |^q \right\} \\
&\stackrel{\text{(ii)}}{\leq} \min_{\alpha \in [0,1]} \left\{ \alpha \langle F'(u^{(n)}), u - u^{(n)} \rangle + \frac{\alpha^q C_{K_0}}{q} | u - u^{(n)} |^q \right\} \\
&\stackrel{\text{(iii)}}{\leq} \min_{\alpha \in [0,1]} \left\{ - \alpha \zeta_n + \frac{\alpha^q C_{K_0}}{q} | u - u^{(n)} |^q \right\},
\end{split}
\end{equation}
where (i) follows from \cref{Lem:stable}, (ii) uses the substitution $w = \alpha (u - u^{(n)})$ for $\alpha \in [0, 1]$, (iii) is due to the convexity of $F$, and $\zeta_n = F(u^{(n)}) - F(u)$.
Both \cref{Thm:conv,Thm:conv_sharp} can be proven by using~\eqref{proof_core}, as presented in the remainder of this appendix.

\subsection{Proof of \cref{Thm:conv}}
We proceed to estimate the last line of~\eqref{proof_core} as follows:
\begin{equation}
\label{proof_conv}
\begin{split}
\min_{\alpha \in [0,1]} \left\{ - \alpha \zeta_n + \frac{\alpha^q C_{K_0}}{q} | u - u^{(n)} |^q \right\}
&\stackrel{\eqref{R_0}}{\leq} \min_{\alpha \in [0,1]} \left\{ - \alpha \zeta_n + \frac{\alpha^q C_{K_0} R_0^q}{q} \right\} \\
&\leq \begin{cases}
- \left(1 - \frac{1}{q} \right) \zeta_n & \text{ if } \zeta_n > C_{K_0} R_0^q, \\
- \left(1 - \frac{1}{q} \right) \frac{\zeta_n^{\frac{q}{q-1}}}{(C_{K_0} R_0^q)^{\frac{1}{q-1}}} & \text{ if } \zeta_n \leq C_{K_0} R_0^q,
\end{cases}
\end{split}
\end{equation}
where the last inequality is due to \cref{Lem:minimum}.
Combining \cref{Lem:descent},~\eqref{proof_core}, and~\eqref{proof_conv}, we obtain
\begin{equation*}
\zeta_{n+1}
\leq \begin{cases}
\left( 1 - \tau \left(1 - \frac{1}{q} \right) \right) \zeta_n & \text{ if } \zeta_n > C_{K_0} R_0^q, \\
\zeta_n - \tau \left(1 - \frac{1}{q} \right) \frac{\zeta_n^{\frac{q}{q-1}}}{(C_{K_0} R_0^q)^{\frac{1}{q-1}}} & \text{ if } \zeta_n \leq C_{K_0} R_0^q.
\end{cases}
\end{equation*}
Note that, by \cref{Cor:descent}, the condition $\zeta_0 \leq C_{K_0} R_0^q$ ensures $\zeta_n \leq C_{K_0} R_0^q$.
Finally, invoking \cref{Lem:recurrence} completes the proof of \cref{Thm:conv}.

\subsection{Proof of \cref{Thm:conv_sharp}}
In the case of \cref{Thm:conv_sharp}, an alternative upper bound for the last line of~\eqref{proof_core} can be derived by invoking \cref{Ass:sharp}:
\begin{equation}
\label{proof_conv_sharp}
\min_{\alpha \in [0,1]} \left\{ - \alpha \zeta_n + \frac{\alpha^q C_{K_0}}{q} | u - u^{(n)} |^q \right\}
\leq \min_{\alpha \in [0,1]} \left\{ - \alpha \zeta_n + \frac{\alpha^q p^{\frac{q}{p}} C_{K_0}}{q \mu_{K_0}^{\frac{q}{p}}} \zeta_n^{\frac{q}{p}} \right\}.
\end{equation}

We first consider the case $p = q$.
It follows by \cref{Lem:minimum} that
\begin{equation}
\label{proof_conv_sharp_p=q}
\begin{split}
\min_{\alpha \in [0,1]} \left\{ - \alpha \zeta_n + \frac{\alpha^q p^{\frac{q}{p}} C_{K_0}}{q \mu_{K_0}^{\frac{q}{p}}} \zeta_n^{\frac{q}{p}} \right\}
&= \min_{\alpha \in [0,1]} \left\{ - \alpha \zeta_n + \frac{\alpha^q q C_{K_0}}{q \mu_{K_0}} \zeta_n \right\} \\
&\leq \zeta_n \left(1 - \frac{1}{q} \right) \min \left\{ 1 , \frac{\mu_{K_0}}{q C_{K_0}} \right\}^{\frac{1}{q-1}}.
\end{split}
\end{equation}
By combining \cref{Lem:descent},~\eqref{proof_core},~\eqref{proof_conv_sharp}, and~\eqref{proof_conv_sharp_p=q}, we obtain the desired result.

Next, we consider the case $p > q$.
By \cref{Lem:minimum}, we have
\begin{multline}
\label{proof_conv_sharp_pnq}
\min_{\alpha \in [0,1]} \left\{ - \alpha \zeta_n + \frac{\alpha^q p^{\frac{q}{p}} C_{K_0}}{q \mu_{K_0}^{\frac{q}{p}}} \zeta_n^{\frac{q}{p}} \right\} \\
\leq \begin{cases}
- \left(1 - \frac{1}{q} \right) \zeta_n & \text{ if } \zeta_n > \left( \frac{p}{\mu_{K_0}} \right)^{\frac{q}{p-q}} C_{K_0}^{\frac{p}{p-q}}, \\
- \left(1 - \frac{1}{q} \right) \left( \frac{\mu_{K_0}}{p} \right)^{\frac{q}{p(q-1)}} \frac{\zeta_n^{\frac{q(p-1)}{p(q-1)}}}{C_{K_0}^{\frac{1}{q-1}}} & \text{ if } \zeta_n \leq \left( \frac{p}{\mu_{K_0}} \right)^{\frac{q}{p-q}} C_{K_0}^{\frac{p}{p-q}},
\end{cases}
\end{multline}
Combining \cref{Lem:descent},~\eqref{proof_core},~\eqref{proof_conv_sharp}, and~\eqref{proof_conv_sharp_pnq}, we get
\begin{equation*}
\zeta_{n+1}
\leq \begin{cases}
\left( 1 - \tau \left(1 - \frac{1}{q} \right) \right) \zeta_n & \text{ if } \zeta_n > \left( \frac{p}{\mu_{K_0}} \right)^{\frac{q}{p-q}} C_{K_0}^{\frac{p}{p-q}}, \\
\zeta_n - \tau \left(1 - \frac{1}{q} \right) \left( \frac{\mu_{K_0}}{p} \right)^{\frac{q}{p(q-1)}} \frac{\zeta_n^{\frac{q(p-1)}{p(q-1)}}}{C_{K_0}^{\frac{1}{q-1}}} & \text{ if } \zeta_n \leq \left( \frac{p}{\mu_{K_0}} \right)^{\frac{q}{p-q}} C_{K_0}^{\frac{p}{p-q}}.
\end{cases}
\end{equation*}
Invoking \cref{Lem:recurrence} completes the proof of \cref{Thm:conv_sharp}.

\section{Triangle inequality-like properties of convex functionals}
\label{App:Triangle}
As presented in \cref{Ass:triangle}, the convergence analysis of nearly semicoercive problems introduced in this paper requires an assumption that each local energy functional satisfies a certain triangle inequality-like property.
To describe this triangle inequality-like property in detail, let $F \colon V \rightarrow \mathbb{R}$ be a G\^{a}teaux differentiable and convex functional defined on a Banach space $V$.
The property states that, for any bounded and convex subset $K$ of $V$, there exists a positive constant $C_{K, \mathrm{tri}}$ such that
\begin{equation}
\label{triangle-like}
d_F (w_1 + w_2; v) \leq C_{K, \mathrm{tri}} \left( d_F (w_1; v) + d_F (w_2; v) \right),
\quad v \in K, \text{ } w_1, w_2 \in V,
\end{equation}
where $d_F$ denotes the Bregman distance associated with $F$ given in~\eqref{d_F}.
In this appendix, we provide several examples of convex functionals that satisfy the triangle inequality-like property~\eqref{triangle-like}.

\begin{example}[quadratic functionals on Hilbert spaces]
\label{Ex:triangle_quadratic}
Let $H$ be a Hilbert space equipped with an inner product $( \cdot, \cdot )$.
We consider the quadratic functional previously given in~\eqref{semidefinite_energy}:
\begin{equation*}
F(v) = \frac{1}{2} ( Av, v ) - ( f, v ),
\quad v \in H,
\end{equation*}
where $A \colon H \rightarrow H$ is a continuous, symmetric and positive semidefinite linear operator, and $f \in H$.
The Bregman distance $d_F$ is given by
\begin{equation*}
d_F (w; v) = \frac{1}{2} ( A w, w ),
\quad v, w \in H.
\end{equation*}
Thanks to the Cauchy--Schwarz inequality, we can readily deduce that~\eqref{triangle-like} holds with $C_{K, \mathrm{tri}} = 2$. 
\end{example}

\begin{example}[smooth and strongly convex functionals]
Let $F$ be an $L$-smooth and $\mu$-strongly convex functional defined on a Banach space $V$, i.e.,
\begin{equation*}
\frac{\mu}{2} \| w \|^2 \leq d_F (w; v) \leq \frac{L}{2} \|w \|^2,
\quad v, w \in V.
\end{equation*}
for some $L, \mu > 0$.
Examples of such smooth and strongly convex functionals can be found in the literature, e.g.,~\cite{CP:2016,CHW:2020,TE:1998}.
For any $v, w_1, w_2 \in V$, it follows that
\begin{equation*}
d_F(w_1 + w_2; v) \leq \frac{L}{2} \| w_1 + w_2 \|^2
\leq L \left( \| w_1 \|^2 + \| w_2 \|^2 \right)
\leq \frac{2L}{\mu} \left( d_F (w_1; v) + d_F(w_2; v) \right),
\end{equation*}
which proves~\eqref{triangle-like}.
\end{example}

\begin{example}[$s$-Laplacian energy]
\label{Ex:triangle_sLap}
Let $\Omega$ be a bounded polyhedral domain in $\mathbb{R}^d$.
We consider the convex functional $F$, defined on $V = W^{1,s} (\Omega)$ with $s > 1$, which arises in variational formulations of $s$-Laplacian problems, as discussed in \cref{Sec:Applications}:
\begin{equation*}
F(v) = \frac{1}{s} \int_{\Omega} | \nabla v |^{s} \,dx,
\quad v \in V.
\end{equation*}
In~\cite[Lemma~3.3]{LP:2024}, it was proven that the Bregman distance $d_F(w; v)$ is equivalent, up to a multiplicative constant, to the squared quasi-norm $\| w \|_{(\nabla v)}^2$ introduced~\cite{BL:1993,LY:2001},  which is given by
\begin{equation*}
\| w \|_{(\nabla v)}^2 = \int_{\Omega} (| \nabla w | + | \nabla v |)^{s-2} | \nabla w |^2 \,dx,
\quad v, w \in V.
\end{equation*}
By leveraging this equivalence relation and the triangle-like property of the quasi-norm established in~\cite[Lemma~5.4]{LY:2001}, we can deduce that the functional $F$ satisfies the triangle-like property~\eqref{triangle-like}.
\end{example}

\bibliographystyle{siamplain}
\bibliography{refs_Schwarz_semi}
\end{document}